\documentclass[12pt]{amsart}

\usepackage{amsmath}
\usepackage{amssymb}
\usepackage{amsthm}
\usepackage{pstricks}
\usepackage{pst-node}

\numberwithin{equation}{section}

\newtheorem{thm}{Theorem}[section]
\newtheorem{prop}[thm]{Proposition}
\newtheorem{conj}[thm]{Conjecture}
\newtheorem{lem}[thm]{Lemma}
\newtheorem{cor}[thm]{Corollary}

\theoremstyle{definition}
\newtheorem{defin}[thm]{Definition}
\newtheorem{rmk}[thm]{Remark}
\newtheorem{ex}[thm]{Example}

\newcommand{\meet}[2]{#1 \cap #2}
\newcommand{\join}[2]{\overleftrightarrow{#1 #2}}

\newcommand{\lattice}[4]{\langle (#1,#2), (#3,#4) \rangle} 

\newcommand{\appendbar}[1]{\multicolumn{1}{c|}{#1}}

\begin{document}
\title{The Devron Property}
\author{Max Glick}
\address{Department of Mathematics, University of Minnesota,
Minneapolis, MN 55455, USA} \email{mglick@umn.edu}



\begin{abstract}
We introduce a criterion called the Devron property that a discrete dynamical system can possess.  The Devron property is said to occur when a class of highly singular inputs of a mapping $F$ are carried by some iterate of $F^{-1}$ to a class of highly singular inputs of $F^{-1}$.  The inspiration for this definition is the discovery by R. Schwartz that the pentagram map exhibits this kind of behavior.  We investigate occurrences of the Devron property in a number of different dynamical systems.  
\end{abstract}

\date{\today}
\thanks{Partially supported by NSF grants DMS-1303482 and DMS-0943745.} 

\maketitle

\tableofcontents
\section{Introduction} \label{secIntro}
Consider a dynamical system obtained by iterating some invertible, algebraic transformation $F$.  If the formulas for $F$ and $F^{-1}$ are rational, then the system will only be defined for part of the state space.  We say that $F$ is \emph{singular} at inputs lying outside its domain of definition.  It is often possible to learn about the system as a whole, by studying the behavior of its singularities.  

Singularity confinement, proposed by B. Grammaticos, A. Ramani, and V. Papageorgiou \cite{GRP}, provides a powerful example of this approach.  Roughly speaking, if the singularities of $F$ are confined (i.e. last only a small number of iterations), then it is likely that $F$ is integrable in some sense.  Moreover, singularity confinement suggests that it should be possible to blow-up, or desingularize, the map obtaining a biregular dynamical system.  The current paper examines a different singularity condition which we name the Devron property.  

Following \cite{GRP}, say that a singularity of $F$ at $x$ is \emph{confined} if there exists a $k>1$ such that
\begin{enumerate}
\item $F^k$ is defined at $x$, say $F^k(x)=y$, and
\item the initial data $x$ can be recovered from $y$.
\end{enumerate}
We consider the question of what happens when moving in the opposite direction from a singularity.  If $F$ is singular at $x$, then one can generally iterate $F^{-1}$ starting from $x$ without problem.  However, for a number of systems and singularity classes we find that $F^{-1}$ itself becomes singular after a predictable, but potentially large, number of steps.  In such cases, we say that the system exhibits the Devron property.  The Devron property is named after a system in the Star Trek universe in which a temporal anomaly propagating backwards in time is discovered.  

The only previous results along these lines of which we are aware concern the pentagram map and its generalizations.  The pentagram map $T$, defined on the space of polygons in the plane, is given by the construction pictured in Figure \ref{figT}.  The vertices of $B=T(A)$ are the intersection points of consecutive shortest diagonals of $A$.  The inventor of the map, R. Schwartz, discovered the following amazing property that it possesses.

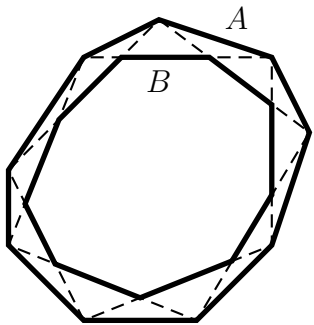
\begin{figure} 
\begin{pspicture}(6,5)
\rput(0,-.5){
  \pspolygon[linewidth=2pt](1,2)(1,3)(2,4.5)(3,5)(4.5,4.5)(5,3.5)(4.5,2)(3.5,1)(2,1)
  \pspolygon[linestyle=dashed](1,2)(2,4.5)(4.5,4.5)(4.5,2)(2,1)(1,3)(3,5)(5,3.5)(3.5,1)
  \pspolygon[linewidth=2pt](1.22,2.55)(1.67,3.67)(2.5,4.5)(3.67,4.5)(4.5,3.87)(4.5,2.67)(3.97,1.79)(2.75,1.3)(1.62,1.75)
  \uput[ur](3.75,4.75){$A$}
  \uput[d](3,4.5){$B$}
}
\end{pspicture}
\caption{An application of the pentagram map with input $A$ and output $B=T(A)$}
\label{figT}
\end{figure}

\begin{thm}[{\cite[Theorem 1.3]{S2}}, {\cite[Theorem 7.6]{G}}]
Let $A$ be an axis-aligned $2m$-gon, as in the left half of Figure \ref{figDevronT}.  Then $B=T^{m-2}(A)$ has its odd vertices all collinear and its even vertices all collinear, as in the right half of Figure \ref{figDevronT}.
\label{thmDevronT}
\end{thm}

\begin{figure}
\begin{pspicture}(10,6)
  \pspolygon(1.5,2)(1.5,2.5)(1,2.5)(1,4)(2,4)(2,4.5)(4,4.5)(4,3)(3,3)(3,2)
  \psline[linestyle=dashed](6.75,.5)(9.25,5.5)
  \psline[linestyle=dashed](9.25,.5)(6.75,5.5)
  \pspolygon(7,5)(8.5,4)(7.5,4)(8.25,3.5)(8.25,2.5)(7.5,2)(8.75,1.5)(7,1)(9,1)(6.75,.5)
\end{pspicture}
\caption{An axis-aligned $2m$-gon (left) and a $2m$-gon with its vertices all lying on two lines (right), where $m=5$.}
\label{figDevronT}
\end{figure}
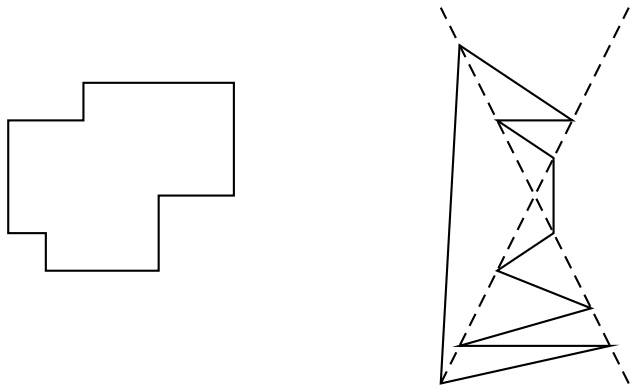

The connection with singularity theory arises from the observation that polygons $B$ as in the Theorem are in a sense as singular as possible for the pentagram map.  Specifically, applying $T$ to $B$ leads to a polygon all of whose vertices are equal (to the point of intersection of the two lines containing the vertices of $B$).  It then becomes impossible to continue iterating $T$.  One can show using Theorem \ref{thmDevronT} that applying $T^{-1}$ to $B$ for $m-2$ steps produces a polygon $A$ that is projectively equivalent to an axis-aligned polygon.  Such $A$ can be considered as singular as possible for $T^{-1}$ because the sides of $T^{-1}(A)$ are all equal.

We make several remarks about Theorem \ref{thmDevronT} in an attempt to illustrate principles that seem to be present in most systems with the Devron property.  First, the result of the Theorem is, for lack of a better word, surprising.  The polygon $A$ is clearly of a very particular type, but applying $T$ a few times leads to one that seems generic to the naked eye.  Imagining that $m$ is large, this will continue for many steps until out of nowhere, another clearly unusual polygon $B$ is reached.

As mentioned, the polygons $A$ and $B$ are as singular as possible for the maps $T^{-1}$ and $T$ respectively.  This is of course not a precise statement, and there is no mention of the severity of singularities in the official definition of the Devron property given in the next section.  However, it seems heuristically that these are the types of singularities for which the Devron property is most likely to be present.

Lastly, we argue that presence of the Devron property can be useful in better understanding the underlying system as a whole.  For instance, consider a polygon $B$ as in Theorem \ref{thmDevronT}.  In \cite{G2} we identify the type of singularity that exists at $B$ as one of the few that are not confined for the pentagram map.  In other words, there is no way to make sense of $T^k(B)$ for any $k$.  Polygons of this sort present a challenge for desingularizing the map.  However, the Devron property says that $T^{-1}$ can only be iterated $m-2$ times starting from $B$, where $m$ is half the number of sides.  Put another way, $B$ is not in the image of $T^k$ for any $k > m-2$.  This fact suggests that a higher iterate of the pentagram map such as $T^m$ could have a more manageable singular structure.

The Devron property for the pentagram map has been generalized in a number of directions.  Schwartz originally proved an analogue pertaining to the map which uses $d$-diagonals (lines connecting vertices $d$ apart) with $d>2$ in place of shortest diagonals \cite{S1}.  In \cite{G}, we extended Theorem \ref{thmDevronT} to a larger class of polygons called twisted polygons.  In addition, Schwartz and Tabachnikov \cite{ST} found results of this sort where the conditions on $A$ and $B$ are that they are circumscribed and inscribed in a conic respectively.  We note that such polygons are not singular for the relevant maps, so this last result is not actually of Devron type.

Despite the number of variations, these results continue to lack a greater context.  The purpose of this paper is to show that the Devron property actually holds in a wide variety of systems including ones with no apparent connection to the pentagram map.  

The first several systems we explore are algebraic in nature, meaning unlike the pentagram map they are defined explicitly via ratios of polynomials.  Specifically, we consider the octahedron recurrence
\begin{displaymath}
f_{i,j,k-1}f_{i,j,k+1} = f_{i-1,j,k}f_{i+1,j,k} -  f_{i,j-1,k}f_{i,j+1,k}
\end{displaymath}
and the $Y$-system recurrence
\begin{displaymath}
Y_{i,j,k+1}Y_{i,j,k-1} = \frac{(1+Y_{i,j-1,k})(1+Y_{i,j+1,k})}{(1+Y_{i-1,j,k}^{-1})(1+Y_{i+1,j,k}^{-1})}.
\end{displaymath}
In both cases, we consider $k$ to be the time direction, and take the variables with two consecutive values of $k$ to be the initial conditions.  To obtain a finite dimensional system, we assume that the initial conditions are periodic with respect to some two dimensional lattice.  

The octahedron recurrence becomes singular if some $f_{i,j,k}=0$, as $f_{i,j,k}$ appears in the denominator when computing $f_{i,j,k+2}$.  We show that if a whole layer of 0's is reached for some $k$, then propagating in the opposite direction eventually leads to another layer of all 0's.  We identify the distance between the layers as a simple function of the periodicity lattice.

The $Y$-system becomes singular if some $Y_{i,j,k}=-1$.  There is an analogous result to that for the octahedron recurrence, namely if a layer of all $-1$'s is encountered in one direction, that another such layer will also be encountered going in the other direction.  In this case, we only prove an upper bound on the distance between the layers.

Using results of \cite{G}, the pentagram map is related by a certain set of coordinates to a case of the $Y$-system.  We deduce a new Devron-type result for the pentagram map.  Using similar reasoning, we consider certain higher dimensional pentagram-like maps which M. Gekhtman, M. Shapiro, S. Tabachnikov, and A. Vainshtein \cite{GSTV} showed also fit within the $Y$-system framework.  Obtaining the Devron property for these systems amounts to identifying the appropriate higher dimensional generalization of an axis-aligned polygon.


We next consider polygon recutting, another geometric dynamical system, which was introduced by V. Adler \cite{A}.  Recutting is defined by cutting a polygon along a shortest diagonal to remove a triangle, and then reattaching the triangle along the same diagonal but with the opposite orientation.  Figure \ref{figRecut} shows an example of this procedure.  This operation is singular if and only if the two vertices used to define the shortest diagonal are equal.

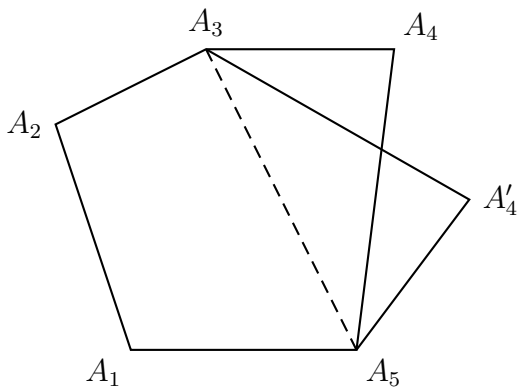
\begin{figure}
\begin{pspicture}(8,6)
\rput(2,0){
\pnode(1,1){A1}
\pnode(0,4){A2}
\pnode(2,5){A3}
\pnode(4.5,5){A4}
\pnode(5.5,3){B4}
\pnode(4,1){A5}

\pspolygon(A1)(A2)(A3)(A4)(A5)
\psline(A3)(B4)(A5)
\psline[linestyle=dashed](A3)(A5)

\uput[dl](A1){$A_1$}
\uput[l](A2){$A_2$}
\uput[u](A3){$A_3$}
\uput[ur](A4){$A_4$}
\uput[r](B4){$A_4'$}
\uput[dr](A5){$A_5$}

}
\end{pspicture}
\caption{An application of polygon recutting}
\label{figRecut}
\end{figure}

We show that the Devron property holds for bipartite polygon recutting, defined as follows.  Let $A$ be a $2n$-gon and color its vertices alternately white and black.  Bipartite recutting means doing recutting at all of the black vertices, then all the white vertices, then all the black vertices, and so on.  The result is that if all of the black vertices of $A$ are equal, then $n-1$ steps of bipartite recutting carries $A$ to another polygon half of whose vertices are equal.

The last system for which we prove the Devron property is a variant of the generalized discrete Toda system studied by S. Iwao \cite{I}.  The system is built from applications of the following transformation to pairs of consecutive columns of a $2 \times n$ matrix
\begin{displaymath}
\left[\begin{array}{cc}
x_1 & y_1 \\
x_2 & y_2 \\
\end{array}
\right]
\mapsto \left[\begin{array}{cc}
y_1\frac{x_2+y_2}{x_1+y_1} & x_1\frac{x_2+y_2}{x_1+y_1} \\
y_2\frac{x_1+y_1}{x_2+y_2} & x_2\frac{x_1+y_1}{x_2+y_2} \\
\end{array}\right]
\end{displaymath}
We identify this transformation as a generalization of polygon recutting at a single vertex.  Using polygon recutting as a model, we obtain the Devron property for the system in which these transformations are carried out in a bipartite manner on a matrix with an even number of columns.

Many of the systems discussed in this paper have a known description in terms of cluster algebras as defined by S. Fomin and A. Zelevinsky \cite{FZ}.  Indeed, any cluster algebra gives rise to a dynamical system and it is reasonable to ask if there are more examples for which the Devron property holds.  Although cluster algebras will often be in the background, we will not explicitly discuss them further.

The remainder of this paper is organized as follows.  Section~ \ref{secDef} gives a formal definition for the Devron property and provides a first example.  In Section~ \ref{secMatrices} we study periodic matrices (i.e. infinite matrices whose entries are periodic with respect to some lattice) proving results needed for Section~ \ref{secOct}.  Sections \ref{secOct} through \ref{secToda} are each devoted to introducing a system or family of systems and proving that the Devron property holds.  The systems are the octahedron recurrence (Section~ \ref{secOct}), the $Y$-system (Section~ \ref{secY}), the pentagram map and its higher/lower dimensional analogues (Section~ \ref{secPenta}), polygon recutting (Section~ \ref{secCut}), and the generalized discrete Toda system (Section~ \ref{secToda}).  The only interdependence is that the results in Section~ \ref{secPenta} use those from Section~ \ref{secY} and results in Section~ \ref{secToda} use those from Section~ \ref{secCut}.  Lastly, Section~ \ref{secConj} discusses a number of systems for which we observe the Devron property, but do not have a proof that it holds.

\medskip

\textbf{Acknowledgments.} I thank Pavlo Pylyavskyy for many helpful suggestions, including pointing me to some of the systems that appear in this paper.  I have enjoyed discussing this subject with Richard Schwartz, and I thank him for his encouragement. 

\section{Main definition} \label{secDef}
We begin by fixing some notation and terminology for rational maps.  For concreteness, we assume here that the underlying space is $\mathbb{C}^n$, but all definitions extend to a suitable class of algebraic varieties.

A \emph{rational map} $F: \mathbb{C}^n \to \mathbb{C}^n$ is an expression of the form
\begin{displaymath}
F = \left(\frac{P_1}{Q_1}, \ldots, \frac{P_n}{Q_n}\right) 
\end{displaymath}
where for each $i=1,\ldots,n$, $P_i$ and $Q_i$ are polynomials in $x_1,\ldots, x_n$ with no nonconstant common factors.  Two rational maps can be composed in the usual way, by substituting and simplifying.  Say that $F$ is \emph{birational} if there exists a rational map $G$ such that $F \circ G$ and $G \circ F$ both equal the identity.  In this case, write $G = F^{-1}$.  If $k \geq 1$ then let
\begin{align*}
F^k &= \underbrace{F \circ F \cdots \circ F}_k \\
F^{-k} &= \underbrace{F^{-1} \circ F^{-1} \cdots \circ F^{-1}}_k
\end{align*}

Let $a = (a_1,\ldots, a_n) \in \mathbb{C}^n$.  Say that $F$ is \emph{defined} at $x=a$ if $Q_i(a) \neq 0$ for all $i$.  In this case, the \emph{value} of $F$ at $a$ is
\begin{displaymath}
F(a) = \left(\frac{P_1(a)}{Q_1(a)}, \ldots, \frac{P_n(a)}{Q_n(a)}\right) \in \mathbb{C}^n.
\end{displaymath}
If some $Q_i(a)=0$ say that $F$ has a \emph{singularity} at $x=a$.  The set of points where $F$ is defined is called the \emph{domain of definition} of $F$.  The set of points where $F$ is singular is called its \emph{singular locus}.  If $U \subseteq \mathbb{C}^n$, let $F(U) \subseteq \mathbb{C}^n$ be the set of values of $F$ at points in $U$ where it is defined.

Note that even if $F$ is birational, it need not be injective on its domain of definition.  For instance, the map $F(x,y) = (y,y/x)$ is birational with inverse $G(x,y) = (x/y,x)$.  However, $F(x,0) = (0,0)$ for all $x \in \mathbb{C} 
\setminus \{0\}$.

\begin{defin} \label{defDevron}
Let $F:X \to X$ be a birational map and suppose $U, V \subseteq X$.  Say that $U,V$ is a \emph{Devron pair} with respect to $F$ if
\begin{itemize}
\item $F^{-a}$ is singular on $U$ for some (small) $a > 0$,
\item $F^{b}$ is singular on $V$ for some (small) $b >0$, 
\item there exists an $M > 0$ such that $F^M(U) \subseteq V$ and $F^{-M}(V) \subseteq U$, and
\item for generic $x \in U$, $F^M$ is defined at $x$ and $F^{-M}$ is defined at $y=F^M(x)$.
\end{itemize}
Call $M$ the \emph{width} of the pair.
\end{defin}

A typical picture, with $a=b=2$ and $M=3$ would be
\begin{displaymath}
F^{-1}(U) \longleftarrow U \longleftrightarrow F(U) \longleftrightarrow F^2(U) \longleftrightarrow V \longrightarrow F(V)
\end{displaymath}
where $\longrightarrow$ denotes a restriction of $F$, $\longleftarrow$ denotes a restriction of $F^{-1}$ and $\longleftrightarrow$ denotes a restriction of $F$ that is birational.

\begin{rmk}
It may appear simpler to consider the pair $U' = F^{-(a-1)}(U)$ and $V' = F^{b-1}(V)$ above.  The problem is that at points that are ``too close'' to the singularity, it might be impossible to propagate in either direction.  For instance, $F$ might not be defined anywhere on $U'$.
\end{rmk}

\begin{rmk}
The first three conditions are enough to ensure that $F$ eventually becomes singular for $x \in U$.  The last condition is needed to ensure that it generically does take $M$ steps (as opposed to fewer) to reach $V$ or encounter another singularity.  The statements that $F^m$ is defined at $x$ and $F^{-m}$ is defined at $F^{m}(x)$ hold for a Zariski open subset of $x \in U$.  As such, to verify the last condition in Definition \ref{defDevron} it is sufficient to find a single $x$ for which these properties hold.
\end{rmk}

The existence of a Devron pair for a system is perhaps not all that surprising.  It could be that for dimension reasons we expect there to be a large class of inputs $V$ such that say $F$ and $F^{-13}$ are both singular at $V$, making it likely that $F^{-12}(V),V$ is a Devron pair.  However, it would not be easy to identify such inputs, since it potentially requires computing $F^{-13}$.  We are interested in the case when $U$ and $V$ are certain natural and easily described subsets of the state space.  If a Devron pair of this sort exists, we say that $F$ exhibits the \emph{Devron property}.

We conclude this section with a simple example of a map exhibiting the Devron property.  Define $F: \mathbb{C}^6 \to \mathbb{C}^6$ by
\begin{equation}
F(a,b,c,d,e,f) = \left(d, \frac{b^2-df}{e}, f, \frac{d^2-fb}{a}, b, \frac{f^2-bd}{c} \right).
\label{eqFirstF}
\end{equation}
This is in fact birational, with inverse
\begin{displaymath}
F^{-1}(a,b,c,d,e,f) = \left(\frac{a^2-ce}{d}, e, \frac{c^2-ea}{f}, a, \frac{e^2-ac}{b}, c\right).
\end{displaymath}
Define $U$ and $V$ by
\begin{align*}
U &= \{(t,b,t,d,t,f)\} \subseteq \mathbb{C}^6 \\
V &= \{(a,t,c,t,e,t)\} \subseteq \mathbb{C}^6 
\end{align*}

\begin{prop}
Given the above, $U,V$ is a Devron pair of width 2.
\end{prop}

\begin{proof}
By direct calculation
\begin{align*}
F^{-1}(t,b,t,d,t,f) &= (0,t,0,t,0,t) \\
F^{-2}(t,b,t,d,t,f) &= (0,0,0,0,0,0)
\end{align*}
It follows that $F^{-3}$ is not defined on $U$.  A similar calculation shows that $F^{3}$ is not defined on $V$.

Now let $(t,b,t,d,t,f) \in U$.  Then
\begin{align*}
F(t,b,t,d,t,f) &= \left(d, \frac{b^2-df}{t}, f, \frac{d^2-fb}{t}, b, \frac{f^2-bd}{t} \right) \\
F^2(t,b,t,d,t,f) &= \left(\frac{d^2-fb}{t}, b', \frac{f^2-bd}{t}, d', \frac{b^2-df}{t}, f' \right)
\end{align*}
where for instance
\begin{align*}
d' &= \frac{\left(\frac{d^2-fb}{t}\right)^2 - \frac{f^2-bd}{t}\frac{b^2-df}{t}}{d} \\
&= \frac{d^4 - 3bd^2f  + b^3d + df^3}{dt^2} = \frac{b^3 + d^3 + f^3 - 3bdf}{t^2}
\end{align*}
This expression is symmetric in $b,d,f$, so by the symmetries of the system we must have $d'=f'=b'$.  Therefore $F^2(t,b,t,d,t,f) \in V$.  So $F^2(U) \subseteq V$, and a similar argument shows $F^{-2}(V) \subseteq U$.

Lastly, let $x=(1,1,1,2,1,3) \in U$.  Then
\begin{align*}
F(x) &= (2,-5,3,1,1,7) \\
F^2(x) &= (1,18,7,18,-5,18) = y.
\end{align*}
There are no zeros in this calculation, so no singularities of $F$ or $F^{-1}$ occur.  Hence $F^2$ is defined at $x$ and $F^{-2}$ is defined at $y$.
\end{proof}

\section{Periodic matrices} \label{secMatrices}
In the next section, we will prove that the octahedron recurrence exhibits the Devron property.  This recurrence has the property that the values computed are given by larger and larger minors of some infinite matrix.  We will focus on the case when the entries of the matrix are periodic.  To understand when singularities occur, it is necessary to understand when periodicity forces a minor to equal zero.  This question is the subject of the current section.

For our purposes, a (two dimensional) \emph{lattice} is a finite index subgroup of $\mathbb{Z}^2$.  If $\Lambda$ is a lattice then it can be expressed as the integer span of two linearly independent integer vectors $(a,b)$ and $(c,d)$.  In this case write $\Lambda = \lattice{a}{b}{c}{d}$.  The determinant of $\Lambda$ is defined to be
\begin{displaymath}
\det(\Lambda) = |\mathbb{Z}^2/\Lambda|
\end{displaymath}
and one can check that $\det(\Lambda) = |ad-bc|$ for $a,b,c,d$ as before.

Say that an infinite matrix $A=(a_{i,j})_{i,j\in\mathbb{Z}}$ is \emph{$\Lambda$-periodic} if $a_{i,j} = a_{i',j'}$ whenever $(i'-i,j'-j) \in \Lambda$.  Let $Mat(\Lambda)$ denote the space of complex, $\Lambda$-periodic matrices.  Note that
\begin{displaymath}
Mat(\Lambda) \cong \mathbb{C}^r
\end{displaymath}
where $r = \det(\Lambda)$.  The same definitions can be made for finite matrices, so let $Mat(\Lambda, m, n)$ denote the space of complex, $\Lambda$-periodic $m\times n$ matrices. 

\begin{prop}
Let $\Lambda \subseteq \mathbb{Z}^2$ be a lattice.  Let $M$ be the minimal positive integer such that either $(M,0) \in \Lambda$ or $(0,M) \in \Lambda$.  Then the consecutive $(M+1)\times(M+1)$ minors of any $\Lambda$-periodic matrix vanish.
\label{propDetEq0}
\end{prop}

\begin{proof}
Let $A \in Mat(\Lambda)$, and suppose without loss of generality that $(M,0) \in \Lambda$.  Then $a_{i+M,j} = a_{i,j}$ for any $i,j \in \mathbb{Z}$.  Therefore, any consecutive $(M+1)\times(M+1)$ submatrix of $A$ has the same first and last row and is hence degenerate.
\end{proof}

Next, we show that there exists a $\Lambda$-periodic matrix with a consecutive $M\times M$ minor that is nonzero.  Again, assume that $(M,0) \in \Lambda$ but that $(i,0) \notin \Lambda$ and $(0,j) \notin \Lambda$ for $0 < |i| < M$ and $0 < |j| < M$.  Choose an element $(k,n)$ of $\Lambda$ with $n>0$ minimal.  For any such choice we have
\begin{displaymath}
\Lambda = \lattice{M}{0}{k}{n}.
\end{displaymath}
If necessary, we can add a multiple of $(M,0)$ to $(k,n)$ to ensure that $0 \leq k < M$.

Assume for the following construction that $M$ is divisible by $k$.  In this case
\begin{displaymath}
\left(0,\frac{Mn}{k}\right) = -(M,0) + \frac{M}{k}(k,n) \in \Lambda
\end{displaymath}
so by assumption $Mn/k \geq M$ and $n \geq k$.

Define $v:\{1,2,\ldots, M\} \to \{1,2,\ldots, k\}$ so that
\begin{displaymath}
i \equiv v(i) \pmod{k}
\end{displaymath}
for all $i$.  Let $i_1,i_2,\ldots,i_M$ be the permutation of $1,2,\ldots, M$ obtained by listing the elements of $v^{-1}(1)$ in increasing order, then the elements of $v^{-1}(2)$ in increasing order, and so on.  

Similarly, define $w:\{1,2,\ldots, M\} \to \{1,2,\ldots, n\}$ so that
\begin{displaymath}
j \equiv w(j) \pmod{n} 
\end{displaymath}
for all $j$.  Let $j_1,j_2,\ldots,j_M$ be the permutation of $1,2,\ldots, M$ obtained by listing the elements of $w^{-1}(1)$ in increasing order, then the elements of $w^{-1}(2)$ in increasing order, and so on.  For example, if $M = 20$, $k=5$ and $n = 9$ then we have
\begin{displaymath}
\begin{array}{|c|cccccccccccccccccccc|}
\hline
t & 1 & 2 & 3 & 4 & 5 & 6 & 7 & 8 & 9 & 10 & 11 & 12 & 13 & 14 & 15 & 16 & 17 & 18 & 19 & 20 \\
\hline
i_t & 1 & 6 & 11 & \appendbar{16} & 2 & 7 & 12 & \appendbar{17} & 3 & 8 & 13 & \appendbar{18} & 4 & 9 & 14 & \appendbar{19} & 5 & 10 & 15 & 20 \\
\hline
j_t & 1 & 10 & \appendbar{19} & 2 & 11 & \appendbar{20} & 3 & \appendbar{12} & 4 & \appendbar{13} & 5 & \appendbar{14} & 6 & \appendbar{15} & 7 & \appendbar{16} & 8 & \appendbar{17} & 9 & 18 \\
\hline
\end{array}
\end{displaymath}

Divide the interval $t=1,2,\ldots, M$ into \emph{verses} based on the value of $v(i_t)$ and into \emph{measures} based on the value of $w(j_t)$.  The divisions in the second row of the preceding table indicate the breaks between verses and the divisions in the third row indicate the breaks between measures.

There are $k$ verses each of size $M/k$.  More precisely, verse $v$ goes from times $t=(M/k)(v-1)+1$ to $t=(M/k)v$ and has values of $i_t$ given by $v,v+k,v+2k,\ldots, v+M-k$.  Going backwards, if $v$ and $i_t$ are known, then $t$ can be computed as
\begin{equation} \label{viTot}
t = \frac{i_t-v}{k} + \frac{M}{k}(v-1) + 1
\end{equation}
a formula that will come in handy later.

There are $n$ measures but they will not all be the same size unless $M$ is divisible by $n$.  In general, the first several will have size $\lceil M/n \rceil$ and the remainder will have size $\lfloor M/n \rfloor$.  Since $n \geq k$, each measure is at most as long as each verse.  As such, each measure is either contained in a verse or is split between two consecutive verses.  In the above example, the second measure ($t=4,5,6$) is the only one that is split between two verses, namely the first verse $(t=1,2,3,4)$ and the second verse $(t=5,6,7,8)$.

Call $(i_1,j_1), \ldots, (i_M,j_M)$ the \emph{melody}.  The \emph{harmony} consists of certain $(i_t \pm 1, j_t)$ according to the following rules:
\begin{itemize}
\item If the measure containing $t$ is split between two verses and $t$ is in the first such verse, then the harmony at time $t$ is $(i_t+1,j_t)$.
\item If the measure containing $t$ is split between two verses and $t$ is in the second such verse, then the harmony at time $t$ is $(i_t-1,j_t)$.
\item If the measure containing $t$ is confined to a single verse then there is no harmony at time $t$.
\end{itemize}
Let $S$ be the set of $(i,j)$  that are part of either the melody or the harmony.  Let $A_{\Lambda}$ be the $M\times M$ matrix whose $(i,j)$-entry is 1 if $(i,j) \in S$ and 0 otherwise.  

In the running example $M=20, k=5,n=9$ there is only harmony during measure 2 ($t=4,5,6$), since that is the only measure split between two verses.  Since $t=4$ is in the earlier verse, the harmony for $(i_4,j_4)= (16,2)$ is $(17,2)$.  Since $t=6,7$ are in the later verse, the harmony for $(i_5,j_5)=(2,11)$ and $(i_6,j_6)= (7,20)$ are $(1,11)$ and $(6,20)$ respectively.  These three points, together with $(i_1,j_1) \ldots (i_{20},j_{20})$ from the table give the locations of the nonzero entries of $A_{\Lambda}$:
\begin{displaymath}
\left(
\begin{array}{cccccccccccccccccccc}
1 & 0 & 0 & 0 & 0 & 0 & 0 & 0 & 0 & 0 & 1' & 0 & 0 & 0 & 0 & 0 & 0 & 0 & 0 & 0 \\
0 & 0 & 0 & 0 & 0 & 0 & 0 & 0 & 0 & 0 & 1 & 0 & 0 & 0 & 0 & 0 & 0 & 0 & 0 & 0 \\
0 & 0 & 0 & 1 & 0 & 0 & 0 & 0 & 0 & 0 & 0 & 0 & 0 & 0 & 0 & 0 & 0 & 0 & 0 & 0 \\
0 & 0 & 0 & 0 & 0 & 1 & 0 & 0 & 0 & 0 & 0 & 0 & 0 & 0 & 0 & 0 & 0 & 0 & 0 & 0 \\
0 & 0 & 0 & 0 & 0 & 0 & 0 & 1 & 0 & 0 & 0 & 0 & 0 & 0 & 0 & 0 & 0 & 0 & 0 & 0 \\
0 & 0 & 0 & 0 & 0 & 0 & 0 & 0 & 0 & 1 & 0 & 0 & 0 & 0 & 0 & 0 & 0 & 0 & 0 & 1' \\
0 & 0 & 0 & 0 & 0 & 0 & 0 & 0 & 0 & 0 & 0 & 0 & 0 & 0 & 0 & 0 & 0 & 0 & 0 & 1 \\
0 & 0 & 0 & 0 & 0 & 0 & 0 & 0 & 0 & 0 & 0 & 0 & 1 & 0 & 0 & 0 & 0 & 0 & 0 & 0 \\
0 & 0 & 0 & 0 & 0 & 0 & 0 & 0 & 0 & 0 & 0 & 0 & 0 & 0 & 1 & 0 & 0 & 0 & 0 & 0 \\
0 & 0 & 0 & 0 & 0 & 0 & 0 & 0 & 0 & 0 & 0 & 0 & 0 & 0 & 0 & 0 & 1 & 0 & 0 & 0 \\
0 & 0 & 0 & 0 & 0 & 0 & 0 & 0 & 0 & 0 & 0 & 0 & 0 & 0 & 0 & 0 & 0 & 0 & 1 & 0 \\
0 & 0 & 1 & 0 & 0 & 0 & 0 & 0 & 0 & 0 & 0 & 0 & 0 & 0 & 0 & 0 & 0 & 0 & 0 & 0 \\
0 & 0 & 0 & 0 & 1 & 0 & 0 & 0 & 0 & 0 & 0 & 0 & 0 & 0 & 0 & 0 & 0 & 0 & 0 & 0 \\
0 & 0 & 0 & 0 & 0 & 0 & 1 & 0 & 0 & 0 & 0 & 0 & 0 & 0 & 0 & 0 & 0 & 0 & 0 & 0 \\
0 & 0 & 0 & 0 & 0 & 0 & 0 & 0 & 1 & 0 & 0 & 0 & 0 & 0 & 0 & 0 & 0 & 0 & 0 & 0 \\
0 & 1 & 0 & 0 & 0 & 0 & 0 & 0 & 0 & 0 & 0 & 0 & 0 & 0 & 0 & 0 & 0 & 0 & 0 & 0 \\
0 & 1' & 0 & 0 & 0 & 0 & 0 & 0 & 0 & 0 & 0 & 1 & 0 & 0 & 0 & 0 & 0 & 0 & 0 & 0 \\
0 & 0 & 0 & 0 & 0 & 0 & 0 & 0 & 0 & 0 & 0 & 0 & 0 & 1 & 0 & 0 & 0 & 0 & 0 & 0 \\
0 & 0 & 0 & 0 & 0 & 0 & 0 & 0 & 0 & 0 & 0 & 0 & 0 & 0 & 0 & 1 & 0 & 0 & 0 & 0 \\
0 & 0 & 0 & 0 & 0 & 0 & 0 & 0 & 0 & 0 & 0 & 0 & 0 & 0 & 0 & 0 & 0 & 1 & 0 & 0 \\
\end{array}
\right)
\end{displaymath}
The entries marked $1'$ indicate contributions of the harmony.

\begin{prop}
$A_{\Lambda}$ is $\Lambda$-periodic.
\end{prop}

\begin{proof}
Recall $\Lambda = \lattice{M}{0}{k}{n}$.  As such, if $(i,j) - (i',j') \in \Lambda$ then $j \equiv j' \pmod{n}$.  Conversely, if $j \equiv j' \pmod{n}$ then there is a unique $i' \in \{1,2,\ldots, M\}$ such that $(i,j) - (i',j') \in \Lambda$.  A coset of $\Lambda$ restricted to $\{1,2,\ldots, M\} \times \{1,2,\ldots, M\}$ can be constructed as follows.  Begin with some point $(i,j)$ in the first $n$ columns and successively add $(k,n)$ until a point on one of the last $n$ columns is reached.  If any step goes past the bottom of the matrix, subtract $(M,0)$ to return to the matrix.

If a measure is confined to a single verse, $(k,n)$ is added to get from each note to the next.  The notes in the measure are exactly a coset as described above.  Now suppose a given measure is split between two verses $v$ and $v+1$.  The notes $(i_t,j_t)$ in the two halves have different residues of $i_t$ modulo $k$, namely $v$ and $v+1$, so they cannot be in the same coset.  The rules for the harmony are designed such that the harmony for the second half of the measure completes the coset containing the melody from the first half, and vice versa.  Together, these notes then comprise two $\Lambda$-cosets.  Put together $S$ is a union of (restrictions to the matrix) of $\Lambda$-cosets, so $A_{\Lambda}$ is $\Lambda$-periodic as desired.
\end{proof}

To illustrate the proof, consider again the example $M=20,k=5,n=9$.  The first measure is contained in the first verse and has notes $(1,1),(6,10),(11,19)$.  These are all related by the vector $(5,9) \in \Lambda$ and no other locations in the matrix lie in the same $\Lambda$-coset.  The second measure is split between the first and second verses.  The melody of this measure is $(16,2), (2,11), (7,20)$, of which only the last two are related by $\Lambda$.  The remainder of the coset of $(16,2)$ consists of $(16,2) + (5,9) - (20,0) = (1,11)$ and $(1,11)+(5,9)=(6,20)$.  These two notes are exactly the harmony for $(2,11)$ and $(7,20)$ as discussed above.  Similarly, the missing element of the coset containing $(2,11)$ and $(7,20)$ is $(2,11) - (5,9) + (20,0) = (17,2)$ which is the harmony for $(16,2)$.

\begin{prop}
$\det(A_{\Lambda}) \neq 0$
\end{prop}

\begin{proof}
First, note that entries $(i_1,j_1),\ldots, (i_M,j_M)$ of $A$ all equal one and $i_1,\ldots, i_M$ and $j_1,\ldots, j_M$ are permutations of $1,\ldots, M$.  This gives a term of $\det(A_{\Lambda})$ equal to $\pm 1$.  The claim is that all other terms of the determinant are zero.

Indeed, $A_{\Lambda}$ is a zero-one matrix in which every column has either a single one or two ones in neighboring rows.  To obtain another nonzero term, it would be necessary to replace some of the ones from the original term with the other ones in the same columns in such a manner that there is still just a single one taken from each row.  It is easy to see that for any such maneuver, there are two consecutive rows where the one in the top row is moved down and the one in the bottom row is moved up.  We will show that no such configuration of four ones occurs in $A_{\Lambda}$.

Indeed, suppose for the sake of contradiction that such a configuration exists.  Let $t,t'$ be such that $i_t$ and $i_{t'}=i_t+1$ are the consecutive rows in question.  The ones in these rows coming from the melody are located at $(i_t,j_t)$ and $(i_{t'},j_{t'})$.  By assumption, there are also ones at $(i_t+1,j_t)$ and $(i_{t'}-1,j_{t'})$, so these must be the harmony for the other two notes.  Let $v$ be the verse containing $t$ and $v'$ the verse containing $t'$.  Since $i_{t'} = i_t+1$ we have $v' \equiv v+1 \pmod{k}$.  Since the harmony of $(i_t,j_t)$ is $(i_t+1,j_t)$, the measure containing $t$ must be split between verses $v$ and $v+1$.  In particular, $v < k$ so $v' = v+1$.  The harmony of $(i_{t'},j_{t'})$ is $(i_{t'}-1,j_t')$ so the measure containing $t'$ must also be split between verses $v'-1=v$ and $v'=v+1$.  Only a single measure can split two given verses, so $t$ and $t'$ are in the same measure.  

Combining $i_{t'} = i_t+1$ and $v' = v+1$ with \eqref{viTot} yields
\begin{displaymath}
t' = t+\frac{M}{k}
\end{displaymath}
so the measure in question has size at least $M/k + 1$.  This contradicts the fact that measures are at most as long as verses, each of which have size $M/k$.  So in fact the determinant has just the one nonzero term, hence $\det(A_{\Lambda}) = \pm 1$.
\end{proof}

The next step is to eliminate the need for the assumption that $k$ divides $M$.

\begin{lem} \label{lemMultiplier}
Let $k$ and $M$ be integers, $d=\gcd(k,M)$.  Then there exists $r \in \mathbb{Z}$ such that $rk \equiv d \pmod{M}$ and $\gcd(r,M)=1$.
\end{lem}

\begin{proof}
Let $r_0,s_0$ be such that $r_0k + s_0M = d$.  Then $r_0k \equiv d \pmod{M}$.  Note that
\begin{displaymath}
r_0\frac{k}{d} + s_0\frac{M}{d} = 1
\end{displaymath}
so $r_0$ and $M/d$ are relatively prime.  Write $M = M_1M_2$ where $M_2$ consists of all prime factors of $M$ that divide $M/d$.  Then $\gcd(M/d,M_1)=1$ so there exists $a\in \mathbb{Z}$ such that $aM/d \equiv 1-r_0 \pmod{M_1}$.  Let $r=r_0+aM/d$.  Then $r \equiv 1 \pmod{M_1}$ so $r$ and $M_1$ are relatively prime.  We still have $r$ and $M/d$ are relatively prime, so $r$ and $M_2$ are relatively prime (since each prime factor of $M_2$ divides $M/d$).  Therefore $\gcd(r,M) = 1$.  Finally
\begin{displaymath}
rk = r_0k + \left(a\frac{k}{d}\right)M \equiv r_0k \equiv d \pmod{M}
\end{displaymath}
as desired.
\end{proof}

\begin{prop}
Let $\Lambda$ be a lattice and let $M>0$ be minimal with respect to the property that $(M,0) \in \Lambda$ or $(0,M) \in \Lambda$.  Then there exists a $\Lambda$-periodic $M\times M$ matrix with nonzero determinant.
\label{propDetNeq0}
\end{prop}

\begin{proof}
Suppose without loss of generality that $(M,0) \in \Lambda$.  As before, suppose $(k,n) \in \Lambda$ with $n>0$ minimal and $0 \leq k < M$.  Then $\Lambda = \lattice{M}{0}{k}{n}$.  Let $d = \gcd(k,M)$ and consider instead the lattice
\begin{displaymath}
\Lambda' = \lattice{M}{0}{d}{n}.
\end{displaymath}

Since $d | M$ the previous construction produces a nondegenerate $M\times M$ matrix $A = A_{\Lambda'} = (a_{ij})$ that is nondegenerate and $\Lambda'$-periodic.  As in Lemma \ref{lemMultiplier} let $r$ be such that $rk \equiv d \pmod{M}$ and $\gcd(r,M)=1$.  Define $B = (b_{ij})$ to be the $M\times M$ matrix with entries
\begin{displaymath}
b_{ij} = a_{ri,j}
\end{displaymath}
where the row index $ri$ into $A$ is understood to be taken modulo $M$.  Then $b_{i+M,j} = a_{ri+rM,j} = a_{ri,j}=b_{i,j}$ and $b_{i+k,j+n} = a_{ri+rk,j+n} = a_{ri+d,j+n} = a_{ri,j} = b_{i,j}$ for all $i$ and $j$.  Therefore $B$ is $\Lambda$-periodic.  Since $\gcd(r,M)=1$, it follows that $B$ is obtained from $A$ by a permutation of the rows.  As such, $B$ is still nondegenerate.
\end{proof}

\begin{cor} \label{corDetNeq0}
Let $\Lambda$ and $M$ be as before.  Then there exists an infinite $\Lambda$-periodic matrix all of whose consecutive $M\times M$ minors are nonzero.
\end{cor}

\begin{proof}
Recall the $Mat(\Lambda) \cong \mathbb{C}^n$ where $n = \det(\Lambda)$.  Because of periodicity, there are only finitely many distinct consecutive $M \times M$ minors of such a matrix.  Each one can individually be arranged to be nonzero by Proposition \ref{propDetNeq0}.  Hence, each is nonzero for a dense open subset of $Mat(\Lambda)$.  It follows that a generic element of $Mat(\Lambda)$ has all of these minors nonzero.
\end{proof}

\section{The octahedron recurrence} \label{secOct}
The octahedron recurrence is
\begin{equation}
f_{i,j,k-1}f_{i,j,k+1} = f_{i-1,j,k}f_{i+1,j,k} + \lambda f_{i,j-1,k}f_{i,j+1,k}.
\label{eqOct}
\end{equation}
Here $\lambda$ is a nonzero constant and the $f_{i,j,k}$ are indexed by $i,j,k \in \mathbb{Z}$ with $i+j+k$ even.  The recurrence has been studied by D. Robbins and H. Rumsey \cite{RR} and D. Speyer \cite{Sp} among others.  Periodic boundary conditions, which we will eventually impose, have been considered by A. Goncharov and R. Kenyon \cite{GK} and by P. Di Francesco and R. Kedem \cite{DK}.  

The case $\lambda=-1$ of \eqref{eqOct} goes back to Dodgson's condensation method for computing determinants.
\begin{prop}[Dodgson's condensation]
Let $\lambda=-1$ and let $(f_{i,j,k})$ be a solution to \eqref{eqOct} with $f_{i,j,0}=1$ for all $i+j$ even.  Then
\begin{displaymath}
f_{i,j,k} = \det\left[ 
\begin{array}{ccccc}
f_{i-k+1,j,1} & f_{i-k+2,j+1,1} & f_{i-k+3,j+2,1} & \cdots & f_{i,j+k-1,1} \\
f_{i-k+2,j-1,1} & f_{i-k+3,j,1} & f_{i-k+4,j+1,1} & \cdots & f_{i+1,j+k-2,1} \\
f_{i-k+3,j-2,1} & f_{i-k+4,j-1,1} & f_{i-k+5,j,1} & \cdots & f_{i+2,j+k-3,1} \\
\vdots & \vdots & \vdots & \ddots & \vdots \\
f_{i,j-k+1,1} & f_{i+1,j-k+2,1} & f_{i+2,j-k+3,1} & \cdots & f_{i+k-1,j,1}\\
\end{array}
\right]
\end{displaymath}
for all $i+j+k$ even with $k \geq 1$.
\label{propDodgson}
\end{prop}

Without the assumption that the $f_{i,j,0}$ equal one, it is still possible to express arbitrary $f_{i,j,k}$ in terms of the $f_{i,j,0}$ and $f_{i,j,1}$ (see \cite{RR}).  However, the expression is more complicated and is no longer given as a determinant.  One can generalize the Dodgson result a bit by exploiting some rescalings of the underlying recurrence.

\begin{prop}
Let $(f_{i,j,k})$ be a solution to \eqref{eqOct} and let $\mu_l,\nu_l$ be nonzero scalars for $l \in \mathbb{Z}$.  Define $g_{i,j,k}=m_{i,j,k}f_{i,j,k}$ for $i+j+k$ even, $k\geq 0$ where
\begin{equation}
m_{i,j,k} = \begin{cases}
\mu_{(i-j)/2}\nu_{(i+j)/2} & k=0 \\
1,& k = 1 \\
\displaystyle\prod_{l=1}^{k-1}\frac{1}{\mu_{(i-j-k)/2+l}\nu_{(i+j-k)/2+l}}, & k \geq 2
 \\
\end{cases}
\label{eqmijk}
\end{equation}
Then the $g_{i,j,k}$ also satisfy \eqref{eqOct}.
\label{propRescale}
\end{prop}

\begin{proof}
A direct calculation shows that 
\begin{displaymath}
m_{i-1,j,k}m_{i+1,j,k} = m_{i,j-1,k}m_{i,j+1,k} = m_{i,j,k-1}m_{i,j,k+1}
\end{displaymath}
for all $i+j+k$ odd with $k \geq 1$.  As a result, substituting the $g_{i,j,k}$ for the $f_{i,j,k}$ in \eqref{eqOct} simply amount to multiplying all three terms by a common constant.
\end{proof}

Note in the octahedron recurrence that
\begin{itemize}
\item The $f_{i,j,k}$ for two consecutive values of $k$ are independent.
\item Such $f_{i,j,k}$ can be used in conjunction with \eqref{eqOct} to determine the $f_{i,j,k}$ for the next larger $k$.
\end{itemize}
As such, we can consider the recurrence as giving rise to a dynamical system in which the state at time $t$ consists of the $f_{i,j,k}$ with $i+j+k$ even and $k = t,t+1$.  These values naturally live in an infinite matrix with entries
\begin{displaymath}
x_{i,j} = \begin{cases}
f_{i,j,t}, &  i+j+t \textrm{ even} \\
f_{i,j,t+1}, & i+j+t \textrm{ odd} \\
\end{cases}
\end{displaymath}
Applying the map replaces the $f_{i,j,t}$ in $x$ with the $f_{i,j,t+2}$ as determined by 
\begin{displaymath}
f_{i,j,t+2} = \frac{f_{i-1,j,t+1}f_{i+1,j,t+1} + \lambda f_{i,j-1,t+1}f_{i,j+1,t+1}}{f_{i,j,t}}.
\end{displaymath}
We focus on the case $\lambda=-1$ as in Dodgson condensation.  It is necessary to keep track of which half of the $x_{i,j}$ change.  Add a bit $\sigma \in \{0,1\}$ for this purpose.  We now have a map $F(x,\sigma) = (x',\sigma')$ where
\begin{align}
\sigma' &= 1-\sigma \label{eqOctsigma} \\
x_{i,j}' &= \begin{cases}
x_{i,j}, & i+j \not\equiv \sigma \pmod{2} \\
\frac{x_{i-1,j}x_{i+1,j} - x_{i,j-1}x_{i,j+1}}{x_{i,j}}, & i+j \equiv \sigma \pmod{2}\\ 
\end{cases} \label{eqOctx} 
\end{align}
This operation is invertible with inverse $F^{-1}(x,\sigma) = (x',1-\sigma)$ where 
\begin{displaymath}
x_{i,j}' = \begin{cases}
x_{i,j}, & i+j \equiv \sigma \pmod{2} \\
\frac{x_{i-1,j}x_{i+1,j} - x_{i,j-1}x_{i,j+1}}{x_{i,j}}, & i+j \not\equiv \sigma \pmod{2}\\ 
\end{cases}
\end{displaymath}

The entries $x_{i,j}$ of $x$ can be rotated 45 degrees and then decomposed into two new infinite matrices $A=A(x,\sigma)$ and $B=B(x,\sigma)$ based on the parity of $i+j$.  Do this in such a manner that $A$ consists of those $x_{i,j}$ that are about to change, i.e. those with $i+j \equiv \sigma \pmod{2}$.  So if $\sigma = 0$ then
\begin{align}
A &= \left[ 
\begin{array}{ccccc}
& \vdots & \vdots & \vdots & \\
\cdots & x_{-2,0} & x_{-1,1} & x_{0,2} & \cdots \\
\cdots & x_{-1,-1} & x_{0,0} & x_{1,1} & \cdots \\
\cdots & x_{0,-2} & x_{1,-1} & x_{2,0} & \cdots \\
& \vdots & \vdots & \vdots & \\
\end{array}
\right] \label{eqA} \\
B &= \left[ 
\begin{array}{cccccc}
& \vdots & \vdots & \vdots & \vdots &  \\
\cdots & x_{-3,0} & x_{-2,1} & x_{-1,2} & x_{0,3} & \cdots \\
\cdots & x_{-2,-1} & x_{-1,0} & x_{0,1} & x_{1,2} & \cdots \\
\cdots & x_{-1,-2} & x_{0,-1} & x_{1,0} & x_{2,1} & \cdots \\
\cdots & x_{0,-3} & x_{1,-2} & x_{2,-1} & x_{3,0} & \cdots \\
& \vdots & \vdots & \vdots & \vdots \\
\end{array}
\right] \label{eqB}
\end{align}
while if $\sigma = 1$ the roles of $A$ and $B$ are swapped.

To obtain a finite dimensional system, assume that $A$ and $B$ are periodic with respect to some lattice $\Lambda \subseteq \mathbb{Z}^2$.  Let 
\begin{equation}
X_{\Lambda} = \{(x,\sigma) : \sigma \in \{0,1\}, A(x,\sigma),B(x,\sigma) \in Mat(\Lambda)\}.
\label{eqXLambda}
\end{equation}
Symmetries of the octahedron recurrence ensure that periodicity will be preserved, so we get a map
\begin{displaymath}
F: X_{\Lambda} \to X_{\Lambda}
\end{displaymath}
given by \eqref{eqOctsigma} and \eqref{eqOctx}.  

Recall that $Mat(\Lambda) \cong \mathbb{C}^n$ where $n = \det(\Lambda)$ so $X_{\Lambda} \cong \mathbb{C}^{2n} \times \{0,1\}$.  In light of \eqref{eqA} and \eqref{eqB}, the condition that $A$ and $B$ are $\Lambda$-periodic is equivalent to requiring that $x$ be periodic with respect to the lattice 
\begin{equation}
\Lambda' = \{(i+j,-i+j) : (i,j) \in \Lambda\}.
\label{eqCompanion}
\end{equation}
We call $\Lambda'$ the \emph{companion} of $\Lambda$.  If $(x,\sigma) \in X_{\Lambda}$ then it gives rise to a solution $f_{i,j,k}$ to the octahedron recurrence such that $f_{i_2,j_2,k} = f_{i_1,j_1,k}$ whenever $(i_2-i_1,j_2-j_1) \in \Lambda'$.

\begin{ex} \label{exOct}
Let $\Lambda = \lattice{3}{0}{1}{1}$.  Then
\begin{displaymath}
\Lambda' = \lattice{3}{-3}{2}{0} = \lattice{1}{-3}{2}{0}
\end{displaymath}
A fundamental domain for $\Lambda'$ is $(1,1),(1,2),\ldots, (1,6)$ so rename the distinct entries $x_{1,1},x_{1,2},\ldots, x_{1,6}$ of $x$ with variables $a,b,c,d,e,f$.  A portion of $x$ looks like
\begin{displaymath}
\left[
\begin{array}{ccccccc}
a & b & c & d & e & f & a \\
d & e & f & a & b & c & d \\
a & b & c & d & e & f & a \\
d & e & f & a & b & c & d \\
\end{array}
\right].
\end{displaymath}
In these coordinates, $F$ looks like
\begin{align*}
F((a,b,c,d,e,f),0) &= \left(\left(\frac{d^2-fb}{a},b,\frac{f^2-bd}{c},d,\frac{b^2-df}{e},f\right),1\right) \\
F((a,b,c,d,e,f),1) &= \left(\left(a,\frac{e^2-ac}{b},c,\frac{a^2-ce}{d},e,\frac{c^2-ea}{f}\right),0\right)
\end{align*}
This agrees with \eqref{eqFirstF} except that in the earlier example reindexing was used to avoid the need for $\sigma$.  Something similar is possible for any $\Lambda$, but we find it easier to use $\sigma$ in the general case.
\end{ex}

Say that an infinite matrix $A$ has rank (at most) 1 if all of its 2-by-2 minors vanish.  The following facts are analogous to ones about finite matrices, and can be proven the same way.

\begin{lem} Let $A$ be an infinite matrix.
\begin{itemize}
\item If $A$ has rank one then its entries can be written in the form $a_{ij}=\mu_i\nu_j$ for some scalars $\mu_i$ and $\nu_j$.
\item Suppose the entries of $A$ are all nonzero.  Then $A$ has rank one if and only if its consecutive 2-by-2 minors are all nonzero.
\end{itemize}
\end{lem}

\begin{thm} \label{thmOct}
Define $U, V \subseteq X_{\Lambda}$ by
\begin{align*}
U &= \{(x,\sigma) : A(x,\sigma) \textrm{ has rank 1}\} \\
V &= \{(x,\sigma) : B(x,\sigma) \textrm{ has rank 1}\}
\end{align*}
Let $M>0$ be minimal with respect to the property $(M,0) \in \Lambda$ or $(0,M) \in \Lambda$.  Then $U,V$ is a Devron pair for $F:X_{\Lambda} \to X_{\Lambda}$ of width $M-1$.
\end{thm}

\begin{proof}
Suppose $(x,\sigma) \in V$, and assume without loss of generality that $\sigma=0$.  Applying $F$ yields 
\begin{displaymath}
x_{i,j}' = \frac{x_{i-1,j}x_{i+1,j} - x_{i,j-1}x_{i,j+1}}{x_{i,j}}
\end{displaymath}
for all $i+j$ even.  The numerator is a 2-by-2 minor of $B(x,\sigma)$ (see \eqref{eqB}) and hence vanishes since $(x,\sigma) \in V$.  So $x_{i,j}'=0$ for all $i+j$ even.  These zeros will be divided by two steps later creating a singularity.  A similar argument shows that $F^{-1}$ quickly becomes singular starting from $U$.

Now suppose $(x,\sigma) \in U$, and assume without loss of generality that $\sigma=0$.  Apply $F$ repeatedly and let $g_{i,j,k}$ denote the resulting solution of \eqref{eqOct} with $\lambda = -1$.  More precisely, let $g_{i,j,0}=x_{i,j}$ for $i+j$ even, let $g_{i,j,1} = x_{i,j}$ for $i+j$ odd, and let the $g_{i,j,k+1}$ be the new values computed during the $k$th iterate of $F$.  Consider the related solution $f_{i,j,k}$ with initial conditions $f_{i,j,0}=1$ for $i+j$ even but $f_{i,j,1} = g_{i,j,1} = x_{i,j}$ for $i+j$ odd.

Since $(x,\sigma) \in U$, we have $A(x,\sigma)$ has rank 1.  In light of \eqref{eqA}, this implies $x_{i+j,-i+j} = \mu_i\nu_j$ for some scalars $\mu_i$ and $\nu_j$.  Changing coordinates,
\begin{displaymath}
x_{i,j} = \mu_{(i-j)/2}\nu_{(i+j)/2}
\end{displaymath}  
for $i+j$ even.  So $g_{i,j,0}=x_{i,j} = m_{i,j,0} = m_{i,j,0}f_{i,j,0}$ for $i+j$ even where $m$ is as defined in \eqref{eqmijk}.  Also, $g_{i,j,1} = f_{i,j,1} = m_{i,j,1}f_{i,j,1}$ for $i+j$ odd.  The $f's$ and $g's$ are both solutions to \eqref{eqOct} so by Proposition \ref{propRescale} and induction
\begin{displaymath}
g_{i,j,k} = m_{i,j,k}f_{i,j,k}
\end{displaymath}
for all $i,j,k$ with $i+j+k$ even, $k \geq 0$.  Combined with Proposition \ref{propDodgson}
\begin{displaymath}
g_{i,j,k} = m_{i,j,k}\det\left[ 
\begin{array}{ccccc}
x_{i-k+1,j} & x_{i-k+2,j+1} & x_{i-k+3,j+2} & \cdots & x_{i,j+k-1} \\
x_{i-k+2,j-1} & x_{i-k+3,j} & x_{i-k+4,j+1} & \cdots & x_{i+1,j+k-2} \\
x_{i-k+3,j-2} & x_{i-k+4,j-1} & x_{i-k+5,j} & \cdots & x_{i+2,j+k-3} \\
\vdots & \vdots & \vdots & \ddots & \vdots \\
x_{i,j-k+1} & x_{i+1,j-k+2} & x_{i+2,j-k+3} & \cdots & x_{i+k-1,j}\\
\end{array}
\right]
\end{displaymath}
In words, each $g_{i,j,k}$ equals a scalar (depending on the $\mu_i$ and $\nu_j$ which depend on $A(x,\sigma)$) times a $k\times k$ consecutive minor of $B(x,\sigma)$.  Generically, the entries of $A(x,\sigma)$ are nonzero which forces each $m_{i,j,k}$ to be a nonzero scalar.  For such generic input, $g_{i,j,k}$ is always defined, and it equals zero if and only the corresponding minor vanishes.

Let $y=F^{M-1}(x)$.  Then $A(y)$ consists of the $g_{i,j,M-1}$ and $B(y)$ consists of the $g_{i,j,M}$.  Now $B(x,\sigma)$ is $\Lambda$-periodic, so by Proposition \ref{propDetEq0}, its consecutive $(M+1)\times(M+1)$ minors vanish.  Therefore, $g_{i,j,M+1} = 0$ for all $i,j$.  By \eqref{eqOct}, 
\begin{displaymath}
g_{i-1,j,M}g_{i+1,j,M} - g_{i,j-1,M}g_{i,j+1,M} = 0
\end{displaymath}
for $i+j+M$ odd.  So the consecutive $2\times 2$ minors of $B(y)$ vanish.  By Corollary \ref{corDetNeq0} there is a choice of $B(x,\sigma)$ whose consecutive $M\times M$ minors are nonzero.  As such the entries of $B(y)$, i.e. the $g_{i,j,M}$, are generically nonzero.  It follows that $B(y)$ has rank 1.  Therefore $y \in V$ as desired.  A similar argument shows that if $(x,\sigma) \in V$ then $F^{-M+1}(x) \in U$.

Lastly, by results of \cite{RR} each $F^k$ is defined at $(x,\sigma) \in X_{\Lambda}$ provided that the entries of $A(x,\sigma)$ are nonzero (this requires the assumption $\lambda=-1$).  If $(x,\sigma) \in U$ has this property then $y=F^{M-1}(x)$ exists.  Moreover, the entries of $B(y)$ are proportional to the $M\times M$ minors of $B(x)$, which are generically nonzero.  In applying $F^{-1}$ the roles of $A$ and $B$ are switched.  Having $B(y)$ nonzero, then, ensures that the iterates of $F^{-1}$ are defined on $y$.
\end{proof}

\begin{rmk}
In general, it is best to consider rational maps defined on irreducible varieties.  This prevents issues that may arise, for instance if the map is singular on an entire reducible component.  As mentioned, the space $X_{\Lambda} \cong \mathbb{C}^{2n} \times \{0,1\}$ and is hence reducible, but there is a symmetry between the two components so this is not a big deal.  On the other hand $U \cong V \cong W \times \{0,1\}$ where $W$ is the space of rank 1, $\Lambda$-periodic matrices, and $W$ can itself be reducible.  For instance, if $\Lambda = \lattice{3}{0}{1}{1}$ as in Example \ref{exOct}, then 
\begin{align*}
((a,b,c,d,e,f),0) \in U &\Longleftrightarrow e^2-ac = a^2-ce = c^2-ae=0 \\
&\Longleftrightarrow \exists t, e=tc=t^2a \textrm{ and } t^3=1
\end{align*}
so $W$ has three components corresponding to the three possible values of $t\in\mathbb{C}$.  A more refined version of Theorem \ref{thmOct} would identify the components of $U$ and $V$ and describe how the map $F^{M-1}$ relates them.
\end{rmk}

\section{$Y$-Systems} \label{secY}
The $Y$-system is similar in spirit, and closely related, to the octahedron recurrence.  Consider variables $Y_{i,j,k}$ for $i+j+k$ even satisfying the relations
\begin{equation}
Y_{i,j,k+1}Y_{i,j,k-1} = \frac{(1+Y_{i,j-1,k})(1+Y_{i,j+1,k})}{(1+Y_{i-1,j,k}^{-1})(1+Y_{i+1,j,k}^{-1})}.
\label{eqYRec}
\end{equation}
It is common to take $1\leq i \leq r$ and $1 \leq j\leq s$ with boundary conditions $Y_{0,j,k}=Y_{r+1,j,k} = \infty$ and $Y_{i,0,k} = Y_{i,s+1,k} = 0$.  The resulting system is periodic in $k$, which is a case of the Zamolodchikov periodicity conjecture \cite{Z}.  We focus on the case where $i,j$ range over all of $\mathbb{Z}$, although we will eventually impose periodic boundary conditions.  A survey discussing $Y$-systems in much greater generality can be found in \cite{KNS}.

Consider the $Y_{i,j,-1}$ and $Y_{i,j,0}$ to be initial conditions.  Specifically, define $u_{i,j}$ by
\begin{displaymath}
u_{i,j} = \begin{cases}
Y_{i,j,-1}^{-1}, & i+j \textrm{ odd} \\
Y_{i,j,0} & i+j \textrm{ even} \\
\end{cases}
\end{displaymath}
Then each $Y_{i,j,k}$ is some rational function of the $u_{i,j}$.  These expressions factor in a predictable way.  Indeed, let
\begin{equation}
M_{i,j,k} = \prod_{l=-k}^k u_{i+l,j}
\label{eqM}
\end{equation}
and recursively define (what will turn out to be) polynomials $F_{i,j,k}$ by $F_{i,j,-1}=F_{i,j,0} = 1$ and
\begin{equation}
F_{i,j,k+1} = \frac{F_{i-1,j,k}F_{i+1,j,k} + M_{i,j,k}F_{i,j-1,k}F_{i,j+1,k}}{F_{i,j,k-1}}
\label{eqF}
\end{equation}
for $k \geq 0$.

\begin{prop}
Expressed in terms of the initial conditions,
\begin{equation}
Y_{i,j,k} = M_{i,j,k}\frac{F_{i,j-1,k}F_{i,j+1,k}}{F_{i-1,j,k}F_{i+1,j,k}}
\label{eqY}
\end{equation}
for all $i+j+k$ even, $k \geq 0$
\end{prop}

\begin{proof}
The base cases $k=0$ and $k=1$ follow easily using $F_{i,j,0}=1$, $M_{i,j,0}=u_{i,j}$, $F_{i,j,1}=1+u_{i,j}$, and $M_{i,j,1} = u_{i-1,j}u_{i,j}u_{i+1,j}$.  Defining the $Y_{i,j,k}$ by \eqref{eqY} we have
\begin{align*}
1+Y_{i,j,k} &= \frac{F_{i,j,k-1}F_{i,j,k+1}}{F_{i-1,j,k}F_{i+1,j,k}} \\
1+Y_{i,j,k}^{-1} &= \frac{F_{i,j,k-1}F_{i,j,k+1}}{M_{i,j,k}F_{i,j-1,k}F_{i,j+1,k}} 
\end{align*}
Both sides of \eqref{eqYRec} become Laurent monomials in certain $M_{i,j,k}$ and $F_{i,j,k}$.  One can check that all the $F$-terms cancel out and what remains is
\begin{displaymath}
M_{i,j,k+1}M_{i,j,k-1} = M_{i-1,j,k}M_{i+1,j,k}
\end{displaymath}
which follows from \eqref{eqM}.
\end{proof}

\begin{rmk}
The general form of the solution \eqref{eqY} is predicted by the cluster algebra method.  In particular, the given proof can be replaced by invoking results of \cite{FZ2}.  Cluster theory also explains why the $F_{i,j,k}$ end up being polynomials as opposed to rational functions.
\end{rmk}

To formally define a dynamical system, consider the state at time $t$ to consist of the $Y_{i,j,t-1}^{-1}$ for $i+j+t$ odd and the $Y_{i,j,t}$ for $i+j+t$ even.  We invert half of the $Y$-variables in order to better match the cluster algebra formulation of $Y$-systems.  As before, we introduce $\sigma \in \{0,1\}$ to keep track of which parity of $i+j$ corresponds to the smaller $t$ value.  Define $G(u,\sigma) = (u', \sigma')$ where
\begin{align}
\sigma' &= 1-\sigma \\
u_{i,j}' &= \begin{cases}
u_{i,j}^{-1}, & i+j \not\equiv \sigma \pmod{2} \\
u_{i,j}\dfrac{(1+u_{i,j-1})(1+u_{i,j+1})}{(1+u_{i-1,j}^{-1})(1+u_{i+1,j}^{-1})}, & i+j \equiv \sigma \pmod{2} \\
\end{cases} \label{eqG}
\end{align}

Define $A=A(u,\sigma)$ and $B=B(u,\sigma)$ as before.  Once again, we assume $A$ and $B$ are periodic with respect to some lattice $\Lambda$.  With these assumptions, we have $G:X_{\Lambda} \to X_{\Lambda}$ with $X_{\Lambda}$ as defined in \eqref{eqXLambda}.

Now, $G$ is singular if $u_{i,j}=-1$ for some $i+j+\sigma$ odd, that is if some entry of $B$ equals $-1$.  As such, let
\begin{align}
U &= \{(u,\sigma) \in X_{\Lambda} : A(u,\sigma) \textrm{ contains all $-1$'s}\} \label{eqUG} \\
V &= \{(u,\sigma) \in X_{\Lambda} : B(u,\sigma) \textrm{ contains all $-1$'s}\} \label{eqVG}
\end{align}
The key to proving the Devron property for the octahedron recurrence was that the formulas for the iterates simplified when restricted to the appropriate set $U$.  The same is true here.  In particular, the polynomials $F_{i,j,k}$ which appear in \eqref{eqY} can be described as determinants of certain matrices whose entries are monomials in the $u_{i,j}$.

To describe these matrices, we need some terminology.  Say that the \emph{negated double ratio} of a $2\times2$ matrix $\left[ \begin{array}{cc} a & b \\ c & d \end{array} \right]$ is $-(bc)/(ad)$.  Given a $k\times k$ matrix $B$, the \emph{lift} of $B$ is the unique $(k+1)\times(k+1)$ matrix $B^*$ with 1's on the diagonal and subdiagonal whose consecutive negated double ratios give the entries of $B$.  That is $b^*_{i,j} = 1$ if $j-i \in \{0,-1\}$ and
\begin{displaymath}
b_{ij} = -\frac{b^*_{i+1,j}b^*_{i,j+1}}{b^*_{i,j}b^*_{i+1,j+1}}
\end{displaymath}
for all $i,j \in \{1,\ldots, n\}$.  For example, the lift of a 2-by-2 matrix is
\begin{displaymath}
B = \left[\begin{array}{cc} a & b \\ c & d \end{array} \right]
\Longrightarrow
B^* = \left[\begin{array}{ccc}
1 & -a & -abd \\
1 & 1 & -d \\
-c & 1 & 1 \\
\end{array}\right]
\end{displaymath}

\begin{prop}
Let $B$ be a $k\times k$ matrix.  Let $C$ be any $(k+1)\times(k+1)$ matrix whose consecutive negated double ratios give the entries of $B$ (no requirement on the diagonal or subdiagonal entries).  Then
\begin{displaymath}
\det(C) = \left(\prod_{i=1}^nc_{ii}\right) \det(B^*)
\end{displaymath}
\label{propLift}
\end{prop}

\begin{lem}
It is possible to rescale the rows and columns of $C$ to obtain $B^*$.
\label{lemRescale}
\end{lem}

\begin{proof}
We want to make the $c_{i,i}$ and $c_{i+1,i}$ all equal to 1.  Order these entries as follows
\begin{displaymath}
c_{1,1}, c_{2,1}, c_{2,2}, c_{3,2}, c_{3,3}, \ldots, c_{n,n}.
\end{displaymath}
Place an order on the rows and columns (leaving out the first row) as follows: first column, second row, second column, third row, third column, \ldots, last row, last column.  Each entry is in the corresponding row or column but not in any subsequent row or column.  So each of these entries can be made to equal 1 by doing rescalings in the indicated order.  The rescalings do not effect the negated double ratios, so they still equal the entries of $B$.  Therefore, after rescaling $C=B^*$.
\end{proof}

\begin{proof}[Proof of Proposition \ref{propLift}]
For any permutation $\pi$ consider the expression
\begin{displaymath}
\frac{\prod_{i=1}^nc_{i,\pi(i)}}{\prod_{i=1}^nc_{i,i}}.
\end{displaymath}
This expression is unchanged by scaling a row or column, so by Lemma \ref{lemRescale} it equals the corresponding expression of $B^*$.  Rearranging:
\begin{align*}
\frac{\prod_{i=1}^nc_{i,\pi(i)}}{\prod_{i=1}^nb^*_{i,\pi(i)}} &= \frac{\prod_{i=1}^nc_{i,i}}{\prod_{i=1}^nb^*_{i,i}} \\
&= \prod_{i=1}^nc_{i,i}
\end{align*}
This is independent of the permutation, so the determinants have the same ratio:
\begin{displaymath}
\frac{\det(C)}{\det(B^*)} = \prod_{i=1}^nc_{i,i}
\end{displaymath}
\end{proof}

\begin{prop}
Let $(u,\sigma) \in U$ and assume for convenience that $\sigma=1$.  Then
\begin{equation}
F_{i,j,k} = \det\left(\left[ 
\begin{array}{cccc}
u_{i-k+1,j} & u_{i-k+2,j+1} & \cdots & u_{i,j+k-1} \\
u_{i-k+2,j-1} & u_{i-k+3,j} & \cdots & u_{i+1,j+k-2} \\
\vdots & \vdots & \ddots & \vdots \\
u_{i,j-k+1} & u_{i+1,j-k+2} & \cdots & u_{i+k-1,j} \\
\end{array}
\right]^*\right)
\label{eqFDet}
\end{equation} 
for $k \geq 0$ and $i+j+k$ odd.  In words, if $A(u,\sigma)$ consists of all $-1$'s, then the $F_{i,j,k}$ are determinants of the lifts of the consecutive $k\times k$ submatrices of $B(u,\sigma)$.
\end{prop}

\begin{proof}
Do induction on $k$.  If $k=0$ then we are taking the lift of a $0\times0$ matrix, which is the matrix $[1]$ whose determinant is $1=F_{i,j,0}$.  Suppose $k=1$.  Then the right hand side is
\begin{displaymath}
\det([u_{i,j}]^*)
= \det\left(\left[ \begin{array}{cc}
1 & u_{i,j}  \\
1 & 1  \\
\end{array}
\right]\right)
= 1+u_{i,j}
\end{displaymath}
which does in fact equal $F_{i,j,1}$.

It remains to verify that the $F_{i,j,k}$ defined by \eqref{eqFDet} satisfy the recurrence \eqref{eqF}.  Suppose $k \geq 1$ and $i+j+k$ is even.  Let $P$ be the matrix for which we are asserting $F_{i,j,k+1} = \det(P^*)$, that is
\begin{displaymath}
P = \left[
\begin{array}{cccc}
u_{i-k,j} & u_{i-k+1,j+1} & \cdots & u_{i,j+k} \\
u_{i-k+1,j-1} & u_{i-k+2,j} & \cdots & u_{i+1,j+k-1} \\
\vdots & \vdots & \ddots & \vdots \\
u_{i,j-k} & u_{i+1,j-k+1} & \cdots & u_{i+k,j} \\
\end{array}
\right] .
\end{displaymath}
The other five $F$-polynomials appearing in \eqref{eqF} are all assumed to be determinants of lifts of submatrices of $P$.  Specifically, $F_{i-1,j,k}$, $F_{i,j+1,k}$, $F_{i+1,j,k}$, and $F_{i,j-1,k}$ arise from the Northwest, Northeast, Southeast, and Southwest $k\times k$ submatrices respectively.  Meanwhile $F_{i,j,k-1}$ arises from the central $(k-1)\times(k-1)$ submatrix.  Let $NW$, $NE$, $SE$, $SW$ denote the corresponding $(k+1)\times(k+1)$ submatrices of the $(k+2)\times(k+2)$ matrix $P^*$.  Let $C$ denote the central $k\times k$ submatrix.  By the Lewis Carroll identity,
\begin{displaymath}
\det(P^*)\det(C) = \det(NW)\det(SE) - \det(NE)\det(SW).
\end{displaymath}
The proof comes down to verifying the following
\begin{align*}
\det(NW) &= F_{i-1,j,k} \\
\det(SE) &= F_{i+1,j,k} \\
\det(NE) &= -M_{i,j,k}F_{i,j+1,k} \\
\det(SW) &= F_{i,j-1,k} \\
\det(C) &= F_{i,j,k-1} 
\end{align*}
Recall that the diagonal and subdiagonal of $P^*$ consist of all 1's.  Because of their alignment, the same is true of $NW$, $SE$, and $C$.  Therefore these three are lifts of the corresponding submatrices of $P$ and the corresponding formulas follow from \eqref{eqFDet}.  In contrast, $NE$ and $SW$ are not lifts.  However their negated double ratios give the corresponding submatrices of $P$, so their determinants can be computed using Proposition \ref{propLift}.  The diagonal entries of $SW$ all equal 1, so we get $\det(SW) = F_{i,j-1,k}$.  The $m$th diagonal entry of $NE$ is $p^*_{m,m+1}$.  We know
\begin{displaymath}
p^*_{m,m} = p^*_{m+1,m+1} = p^*_{m+1,m} = 1
\end{displaymath}
so to get the correct negated double ratio
\begin{displaymath}
p^*_{m,m+1} = -p_{m,m} = u_{i-k+2(m-1),j}
\end{displaymath}
Multiplying all of these yields
\begin{displaymath}
\det(NE) = (-1)^{k+1}\left(\prod_{m=1}^{k+1}u_{i-k+2(m-1),j}\right)F_{i,j+1,k}
\end{displaymath}
Recall $M_{i,j,k} = u_{i-k,j}u_{i-k+1,j}\cdots u_{i+k,j}$.  We are assuming $u_{i,j}=-1$ for $i+j$ odd.  Since $i+j+k$ is even
\begin{align*}
M_{i,j,k} &= u_{i-k,j}(-1)u_{i-k+2,j}(-1)\cdots (-1)u_{i+k,j} \\
&= (-1)^k\left(\prod_{m=1}^{k+1}u_{i-k+2(m-1),j}\right)
\end{align*}
So $\det(NE) = -M_{i,j,k}F_{i,j+1,k}$ as desired.
\end{proof}

Experiments suggest that the map $G$ driving the $Y$-system does not exhibit the Devron property on $X_{\Lambda}$.  More precisely, for any $k > 0$ we have that $G^k$ is defined for generic $(u,\sigma) \in U$ with $U$ as in $\eqref{eqUG}$.  However, there is an invariant hypersurface in $X_{\Lambda}$ on which the Devron property holds.  Define a map $\rho:X_{\Lambda} \to \mathbb{C}$ by
\begin{displaymath}
\rho(u,\sigma) = \prod_{i,j} u_{i,j}
\end{displaymath}
where the product is over the distinct $u_{i,j}$.  Recall that the $u_{i,j}$ are periodic modulo the lattice
\begin{displaymath}
\Lambda' = \{(i+j,-i+j) : (i,j) \in \Lambda\}
\end{displaymath}
which has index $2n$ where $n = \det(\Lambda)$.  So $\rho$ is a product of $2n$ variables.

\begin{lem}
$\rho$ is a conserved quantity of $G$, i.e. $\rho(G(u,\sigma)) = \rho(u,\sigma)$ for all $(u,\sigma) \in X_{\Lambda}$.
\end{lem}

\begin{proof}
Let $(u',1-\sigma) = G(u,\sigma)$.  Break the expression $\rho(u',1-\sigma)$ into two pieces, one for each parity of $i+j$.  By \eqref{eqG},
\begin{displaymath}
\prod_{i+j+\sigma \textrm{ odd}} u'_{i,j} = \prod_{i+j+\sigma \textrm{ odd}} u_{i,j}^{-1}.
\end{displaymath}
Also by \eqref{eqG},
\begin{displaymath}
\prod_{i+j+\sigma \textrm{ even}} u'_{i,j} = \left(\prod_{i+j+\sigma \textrm{ even}} u_{i,j} \right) \frac{\displaystyle \prod_{i+j+\sigma \textrm{ odd}} (1+u_{i,j})^2}{\displaystyle \prod_{i+j+\sigma \textrm{ odd}} (1+u_{i,j}^{-1})^2}
\end{displaymath}
Using $(1+x)/(1+x^{-1}) = x$ this simplifies to 
\begin{displaymath}
\prod_{i+j+\sigma \textrm{ even}} u'_{i,j} = \prod_{i+j+\sigma \textrm{ even}} u_{i,j}  \prod_{i+j+\sigma \textrm{ odd}} u_{i,j}^2
\end{displaymath}
Therefore
\begin{align*}
\rho(u',1-\sigma) &= \prod_{i+j+\sigma \textrm{ odd}} u'_{i,j}\prod_{i+j+\sigma \textrm{ even}} u'_{i,j} \\
&= \prod_{i+j+\sigma \textrm{ odd}} u_{i,j}^{-1} \prod_{i+j+\sigma \textrm{ even}} u_{i,j}  \prod_{i+j+\sigma \textrm{ odd}} u_{i,j}^2 \\
&= \prod_{i,j} u_{i,j} = \rho(u,\sigma)
\end{align*}
as desired.
\end{proof}

Define 
\begin{equation}
\overline{X}_{\Lambda} = \{(u,\sigma)\in X_{\Lambda} : \rho(u,\sigma)=1\}.
\label{eqXLambdaBar}
\end{equation}
Since $\rho$ is conserved by $G$, we get a restriction
\begin{displaymath}
G : \overline{X}_{\Lambda} \to \overline{X}_{\Lambda}.
\end{displaymath}
Define $\overline{U} = \overline{X}_{\Lambda} \cap U$ and $\overline{V} = \overline{X}_{\Lambda} \cap V$ for $U,V$ as in \eqref{eqUG} and \eqref{eqVG}.  Recall that $A(u,\sigma)$ has all entries equal to -1 for $(u,\sigma) \in U$.  Hence the product of these entries is $(-1)^n$.  As such, an element  $(u,\sigma) \in U$ is in $\overline{U}$ if and only if the product of the distinct entries of $B(u,\sigma)$ also equal $(-1)^n$.

\begin{prop} \label{propFEq0}
Let $(u,\sigma) \in \overline{U} \subseteq \overline{X}_{\Lambda}$ and assume without loss of generality that $\sigma=1$.  Then $F_{i,j,n}=0$ for all $i+j+n$ odd where $n = \det(\Lambda)$.
\end{prop}

\begin{proof}
Let $P$ be the consecutive $n\times n$ submatrix of $B(u,\sigma)$ such that
\begin{displaymath}
F_{i,j,k} = \det(P^*)
\end{displaymath}
according to \eqref{eqFDet}.  Let $a>0$ be minimal with respect to the property that $(a,0) \in \Lambda$.  It follows that $a | n$ and that 
\begin{displaymath}
\{(i,j) : 1 \leq i \leq a, 1 \leq j \leq b\}
\end{displaymath}
is a fundamental domain for $\Lambda$, where $b=n/a$.  Since $(u,\sigma) \in \overline{U}$,
\begin{displaymath}
\prod_{i=1}^{a}\prod_{j=1}^{b} p_{i,j} = (-1)^n
\end{displaymath}
or put another way, the product of the $-p_{i,j}$ in the rectangle equals 1.  Recall 
\begin{displaymath}
-p_{i,j} = \frac{p^*_{i+1,j}p^*_{i,j+1}}{p^*_{i,j}p^*_{i+1,j+1}}.
\end{displaymath}
Computing the product, all the $p^*_{i,j}$ cancel out except those on the corners, leaving
\begin{displaymath}
1 = \frac{p^*_{a+1,1}p^*_{1,b+1}}{p^*_{1,1}p^*_{a+1,b+1}}.
\end{displaymath}
The same argument works with any $a\times b$ rectangle, so 
\begin{displaymath}
1 = \frac{p^*_{a+i,j}p^*_{i,b+j}}{p^*_{i,j}p^*_{a+i,b+j}}.
\end{displaymath}
for any $i,j$.  In fact, the same holds starting from any rectangles whose dimensions are divisible by $a$ and $b$ respectively, so
\begin{displaymath}
1 = \frac{p^*_{i',j}p^*_{i,j'}}{p^*_{i,j}p^*_{i',j'}}.
\end{displaymath}
whenever $i \equiv i' \pmod{a}$ and $j \equiv j' \pmod{b}$.  Put another way, 
\begin{displaymath}
0 = p^*_{i',j}p^*_{i,j'} -{p^*_{i,j}p^*_{i',j'}}.
\end{displaymath}
for such $i,i',j,j'$.

Now $P^*$ is and $(n+1)\times(n+1)$ matrix, where $n=ab$.  We will show that the $b+1$ rows $I = \{1,a+1,2a+1,\ldots, n+1\}$ are linearly dependent.  Indeed, consider any maximal minor of these rows, i.e. a $(b+1)\times(b+1)$ minor of $P^*$ that consists of the rows in $I$ and any $b+1$ columns $J$.  By the pigeonhole principle, there exists $j,j' \in J$ with $j \equiv j' \pmod{b}$.  Therefore the 2-by-2 minors of
\begin{displaymath}
\left[\begin{array}{cc}
p^*_{1,j} & p^*_{1,j'} \\
p^*_{a+1,j} & p^*_{a+1,j'} \\
\vdots & \vdots \\
p^*_{n+1,j} & p^*{n+1,j'} \\
\end{array} \right]
\end{displaymath} 
all equal 0.  It follows that the $I,J$ minor of $P^*$ equals 0.  As $J$ was arbitrary, this confirms that the rows $I$ are dependent.  Therefore $P^*$ is degenerate and $F_{i,j,n} = \det(P^*) = 0$. 
\end{proof}

\begin{thm}
Let $G: \overline{X}_{\Lambda} \to \overline{X}_{\Lambda}$ and $\overline{U}, \overline{V}$ all be as before.  Then $\overline{U}, \overline{V}$ is a Devron pair for $G$.  The width of the pair is at most $n-1$ where $n = \det(\Lambda)$.
\label{thmYDevron}
\end{thm}

\begin{proof}
It is clear from \eqref{eqG} that $G$ is singular when $u_{i,j}=-1$ for the appropriate parity of $i+j$.  The sets $\overline{U}$ and $\overline{V}$ are defined by certain equations $u_{i,j}=-1$ ensuring a singularity for $G^{-1}$ and $G$ respectively.

By Proposition \ref{propFEq0}, we have that $F_{i,j,n}$ vanishes for $i+j+n$ odd when restricted to $\overline{U}$.  Let $N>0$ be minimal with respect to the property that $F_{i,j,N+1}$ vanishes on $\overline{U}$ for $i+j+N$ even.  Then $N \leq n-1$.  

Let $(u,\sigma) \in \overline{U}$.  Assume without loss of generality that $\sigma = 1$.  Let $Y_{i,j,k}$ be the corresponding solution to \eqref{eqYRec}, i.e. the one with initial conditions $Y_{i,j,-1} = u_{i,j}^{-1}$ for $i+j$ odd and $Y_{i,j,0} = u_{i,j}$ for $i+j$ even.  Let $v = G^N(u)$, which consists of the $Y_{i,j,N-1}$ and $Y_{i,j,N}$.  Then
\begin{displaymath}
\begin{array}{rcll}
0 &= &F_{i-1,j,N}F_{i+1,j,N} + M_{i,j,N}F_{i,j-1,N}F_{i,j+1,N} & \textrm{ (by \eqref{eqF})} \\
-1 &= &\dfrac{M_{i,j,N}F_{i,j-1,N}F_{i,j+1,N}}{F_{i-1,j,N}F_{i+1,j,N}} & \\
&= &Y_{i,j,N} & \textrm{ (by \eqref{eqY})} \\
\end{array}
\end{displaymath}
for all $i+j \equiv N \pmod{2}$.  The division here is justified because minimality of $N$ ensures that the $F_{i,j,N}$ are generically nonzero.  So $v \in \overline{V}$ as desired.

Suppose $(u,\sigma) \in \overline{U}$ is generic.  As just mentioned, the choice of $N$ guarantees that the $F_{i,j,k}$ are nonzero for $k \leq N$.  By \eqref{eqY}, the $Y_{i,j,k}$ are defined for $k \leq N$ so $G^N$ is defined at $(u,\sigma)$.  Moreover, the fact that none of the $F_{i,j,k}$ vanish implies that none of the $Y_{i,j,k}=-1$ for $0 \leq k < n$.  Hence each application of $G$ is defined and reversible.  In particular, $G^N$ is defined on $(u,\sigma)$ and $G^{-N}$ is defined on $G^N(u,\sigma)$.
\end{proof}

\begin{rmk}
Experimental evidence suggests that the width of the Devron pair in Theorem \ref{thmYDevron} is in fact $n-1$, but all that is proven is that $n-1$ is an upper bound.  In words, the Theorem asserts that a layer of $-1$'s in the $Y$-system always reappears after $n-1$ steps or fewer.  To prove the width is exactly $n-1$, it would suffice to find an example input for each lattice $\Lambda$ for which it does actually take $n-1$ steps for this to happen. 
\end{rmk}

\section{The pentagram map and generalizations} \label{secPenta}

\subsection{The pentagram map}
The pentagram map, introduced by Schwartz \cite{S0} is an operation defined for polygons in the plane.  If the original polygon has vertices $A_i$ then the pentagram map takes it to one with vertices
\begin{displaymath}
\meet{\join{A_i}{A_{i+2}}}{\join{A_{i+1}}{A_{i+3}}}.
\end{displaymath}
In words, each new vertex is at the intersection of two consecutive ``shortest'' diagonals of the old polygon (see Figure \ref{figT}).  The pentagram map is denoted $T$.  We will prove a new instance of the Devron property for the pentagram map, adding to preexisting results mentioned in the introduction.

The pentagram map is often considered for twisted polygons.  A \emph{twisted $n$-gon} is a sequence $A=(A_i)_{i \in \mathbb{Z}}$ of points in the projective plane, together with a projective transformation $\phi$ such that
\begin{displaymath}
A_{i+n} = \phi(A_i)
\end{displaymath}
for all $i \in \mathbb{Z}$.  Additionally, assume for each $i$ that the points $A_i, A_{i+1}, A_{i+2}, A_{i+3}$ are in general position (i.e. no three are collinear).  The map $\phi$ is called the \emph{monodromy} of the polygon.  If $\phi$ is the identity then the polygon is called \emph{closed}.  Let $\mathcal{P}_n$ be the space of twisted $n$-gons and $\mathcal{C}_n$ the space of closed $n$-gons.  

When the pentagram map is applied to a twisted polygon, the result may no longer have the property that nearby vertices are in general position.  As such $T$ is only defined generically on $\mathcal{P}_n$.  It maps a twisted polygon to one with the same monodromy.  As such we have rational maps
\begin{displaymath}
T: \mathcal{P}_n \to \mathcal{P}_n
\end{displaymath}
and
\begin{displaymath}
T: \mathcal{C}_n \to \mathcal{C}_n.
\end{displaymath}
Generically, the operation is invertible where the inverse $T^{-1}$ takes a polygon $B$ to one with vertices
\begin{displaymath}
\meet{\join{B_{i}}{B_{i+1}}}{\join{B_{i+2}}{B_{i+3}}}.
\end{displaymath}
The inverse map is illustrated in Figure \ref{figTInv}.

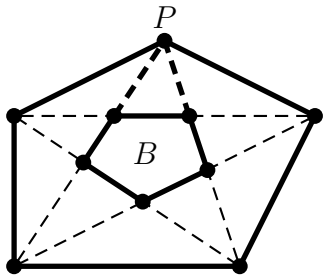
\begin{figure} \label{figTInv}
\begin{pspicture}(5,4)
\rput(0,-.5){
  \pspolygon[showpoints=true,linewidth=2pt](1,1)(4,1)(5,3)(3,4)(1,3)
  \uput[90](3,4){$P$}
  \psline[linestyle=dashed](1.92,2.38)(1,1)(5,3)(1,3)(4,1)(3.57,2.28)
  \pspolygon[showpoints=true,linewidth=2pt](1.92,2.38)(2.71,1.86)(3.57,2.28)(3.33,3)(2.33,3)
  \psline[linestyle=dashed,linewidth=2pt](2.33,3)(3,4)(3.33,3)
  \rput(2.75,2.5){$B$}
}
\end{pspicture}

\caption{A typical vertex $P$ of $T^{-1}(B)$ lies at the intersection of (the extensions of) two sides of $B$.}
\end{figure}

\begin{rmk}
There is a difficulty in indexing when defining the pentagram map.  For instance, if $B=T(A)$ then one must decide to which vertex of $B$ to assign $\meet{\join{A_1}{A_3}}{\join{A_2}{A_4}}$.  One might reasonably choose
\begin{itemize}
\item $B_2$ (the left indexing scheme)
\item $B_3$ (the right indexing scheme) or
\item $B_{2.5}$.
\end{itemize}
The last option is the most symmetric, but leads to an awkward situation in which the polygons alternate each step between being indexed by $\mathbb{Z}$ and $\frac{1}{2} + \mathbb{Z}$.  Our results will not depend on which method is used, so we do not specify one.
\end{rmk}

In \cite{G}, we connected the pentagram map to a certain $Y$-system which can be seen as a specialization of the system considered in the Section \ref{secY}.  If $A$ is a twisted polygon, then certain cross ratios computed from the vertices of the $T^k(A)$ satisfy a family of relations similar to \eqref{eqYRec}.  More precisely, let $A$ be a twisted polygon and let $P_{2j,0} = A_j$ for all $j \in \mathbb{Z}$.  Extend this to an array of points $P_{i,k}$ with $i+k$ even by letting
\begin{displaymath}
P_{i,k} = \meet{\join{P_{i-3,k-1}}{P_{i+1,k-1}}}{\join{P_{i-1,k-1}}{P_{i+3,k-1}}}
\end{displaymath}
for $k > 0 $ and 
\begin{displaymath}
P_{i,k} = \meet{\join{P_{i-3,k+1}}{P_{i-1,k+1}}}{\join{P_{i+1,k+1}}{P_{i+3,k+1}}}
\end{displaymath}
for $k < 0$.  Defined this way, the $P_{i,k}$ are precisely the vertices of $T^k(A)$ for all $k \in \mathbb{Z}$.

Fix $i,k$ with $i+k$ even.  Then diagonal $\join{P_{i-2,k}}{P_{i+2,k}}$ of $T^k(A)$ is used to construct both $P_{i-1,k+1}$ and $P_{i+1,k+1}$.  Hence these four points are collinear.  Define $y_{i,k} = y_{i,k}(A)$ to be their negated cross ratio
\begin{displaymath}
y_{i,k} = -\chi(P_{i-2,k}, P_{i-1,k+1}, P_{i+1,k+1}, P_{i+2,k}).
\end{displaymath}
Here we are defining the cross ratio with the convention 
\begin{displaymath}
\chi(a,b,c,d) = \frac{(a-b)(c-d)}{(a-c)(b-d)}
\end{displaymath}
for scalars $a,b,c,d$ and extending in the natural way to the case of four collinear points in the plane.

\begin{prop}[{\cite[Proposition 2.3]{G}}]
Let $A$ be a generic twisted polygon and compute the $y_{i,k} = y_{i,k}(A)$ as above.  If $i,k \in \mathbb{Z}$ and $i+k$ is odd then
\begin{displaymath}
y_{i,k-1}y_{i,k+1} = \frac{(1+y_{i-3,k})(1+y_{i+3,k})}{(1+y_{i-1,k}^{-1})(1+y_{i+1,k}^{-1})}
\end{displaymath}
Put another way, if $Y_{i,j,k} = y_{i+3j,k}$ for all $i+j+k$ even then the $Y_{i,j,k}$ satisfy the $Y$-system \eqref{eqYRec}.
\end{prop}

As such, all the $y_{i,k}$ can be determined by the $y_{i,-1}$ and the $y_{i,0}$.  Moreover, these $y$-values are nearly independent.  Define $u_i(A)$ for $i \in \mathbb{Z}$ by
\begin{displaymath}
u_i(A) = \begin{cases}
y_{i,-1}(A)^{-1}, & i \textrm{ odd} \\
y_{i,0}(A), & i \textrm{ even} \\
\end{cases}
\end{displaymath}

\begin{prop}[\cite{G}] 
Let $A$ be a twisted polygon, $u_i = u_i(A)$.  Then
\begin{itemize}
\item $u_{i+2n} = u_i$ for all $i \in \mathbb{Z}$.
\item $u_1u_2\cdots u_{2n} = 1$ but $u_1,u_2,\ldots u_{2n}$ satisfy no other relations.
\item The values $u_1,u_2,\ldots, u_{2n-1}$ are enough to reconstruct $A$ up to a projective transformation and the rescaling operation introduced in \cite{S2}.  Hence the $u_i$ give coordinates on the space of twisted polygons modulo these operations.
\end{itemize}
\end{prop}

\begin{rmk}
The rescaling operation mentioned above plays a central role in establishing integrability of the pentagram map (see \cite{OST}).  For our purposes, however, it is not necessary to keep track of the rescaling parameter.  More precisely, the spaces $U$ and $V$ that make up the Devron pair for the pentagram map will both be invariant under rescaling.
\end{rmk}

Let $Y_{i,j,k} = y_{i+3j,k}$.  Then $Y_{i,j,-1}^{-1} = y_{i+3j,-1}^{-1} = u_{i+3j}$ for $i+j$ odd and $Y_{i,j,0} = y_{i+3j,0} = u_{i+3j}$ for $i+j$ even.  As such, define $u_{i,j} = u_{i+3j}$ for all $i,j \in \mathbb{Z}$.  Part of the $u$-matrix looks like
\begin{displaymath}
\left[
\begin{array}{ccc}
u_0 & u_3 & u_6 \\
u_1 & u_4 & u_7 \\
u_2 & u_5 & u_8 \\
\end{array}
\right].
\end{displaymath}
Decomposing and rotating as in \eqref{eqA} and \eqref{eqB} yields matrices
\begin{displaymath}
\left[
\begin{array}{ccccc}
& \vdots & \vdots & \vdots & \\
\cdots & u_2 & u_6 & u_{10} & \cdots \\
\cdots & u_0 & u_4 & u_8 & \cdots \\
\cdots & u_{-2} & u_2 & u_6 & \cdots \\
& \vdots & \vdots & \vdots & \\
\end{array}
\right]
\end{displaymath}
and
\begin{displaymath}
\left[
\begin{array}{cccc}
& \vdots & \vdots & \\
\cdots & u_3 & u_7 & \cdots \\
\cdots & u_1 & u_5 & \cdots \\
& \vdots & \vdots & \\
\end{array}
\right].
\end{displaymath}

By inspection, these matrices are periodic with respect to the vector $(2,1)$.  Also, since $u_{i+2n} = u_i$ for all $i \in \mathbb{Z}$ they are periodic with respect to $(n,0)$.  As such, $(u,\sigma) \in X_{\Lambda}$ for $X_{\Lambda}$ as defined in \eqref{eqXLambda}, $\sigma=1$, and 
\begin{displaymath}
\Lambda = \lattice{2}{1}{n}{0}.
\end{displaymath}
The $2n$ distinct entries are $u_2,u_4,\ldots, u_{2n}$ together with $u_1,u_3,\ldots, u_{2n-1}$ so 
\begin{displaymath}
\rho(u,\sigma) = u_1u_2u_3\cdots u_{2n} = 1.
\end{displaymath}
Therefore $(u,\sigma) \in \overline{X}_{\Lambda}$ with $\overline{X}_{\Lambda}$ as defined in \eqref{eqXLambdaBar}.

To sum up, the pentagram map on twisted $n$-gons can be described in certain coordinates as the map $G$ restricted to $\overline{X}_{\Lambda}$ where $\Lambda = \lattice{2}{1}{n}{0}$.  By Theorem \ref{thmYDevron}, the Devron property holds for $G$ so the same result follows for the pentagram map.  The singular sets $U$ and $V$ have a simple geometric interpretation.  Recall $(u,\sigma) \in U$ if and only if $u_{i,j} = -1$ for all $i+j \equiv \sigma \pmod{2}$, and $V$ is defined by setting the other half of the $u_{i,j}$ to $-1$.  To translate this to polygons, define $U,V \subseteq \mathcal{P}_n$ by
\begin{align*}
U &= \{A \in \mathcal{P}_n : u_{2i+1}(A) = -1 \textrm{ for all } i\in \mathbb{Z}\} \\
V &= \{A \in \mathcal{P}_n : u_{2i}(A) = -1 \textrm{ for all } i \in \mathbb{Z}\}
\end{align*}

\begin{prop}[{\cite[Lemma 7.2]{G}}]
Let $A \in \mathcal{P}_n$.  Then
\begin{itemize}
\item $A \in U$ if and only if the sides $\join{A_{2i}}{A_{2i+1}}$ for $i \in \mathbb{Z}$ are concurrent and the sides $\join{A_{2i-1}}{A_{2i}}$ for $i \in \mathbb{Z}$ are also concurrent.
\item $A \in V$ if and only if the vertices $A_{2i}$ for $i \in \mathbb{Z}$ are collinear and the vertices $A_{2i+1}$ for $i \in \mathbb{Z}$ are also collinear.
\end{itemize}
\label{propAxisAligned}
\end{prop}

Say that $A$ is \emph{axis-aligned} if half of its sides pass through one point $P$ and the other half pass through another point $Q$ as above.  The reason for the terminology is that if we choose $P$ and $Q$ as certain points at infinity then the sides of $A$ will alternate between being parallel to the $x$ and $y$ axes.  Say that $A$ is \emph{dual axis-aligned} if half its sides lie on one line and the other half lie on another.  In light of Proposition \ref{propAxisAligned} we can switch to geometric definitions of $U$ and $V$.

\begin{thm}
Define $U,V \subseteq \mathcal{P}_n$ by 
\begin{align*}
U &= \{A : A \textrm{ axis-aligned}\} \\
V &= \{A : A \textrm{ dual axis-aligned}\}
\end{align*}
Then $(U,V)$ is a Devron pair for the pentagram map.  The width of the pair is at most $n-1$.
\label{thmPenta}
\end{thm}
\begin{proof}
In light of the identification of the pentagram map with a $Y$-system, this is just a specialization of Theorem \ref{thmYDevron}.  The periodicity lattice used in the $Y$-system is $\lattice{2}{1}{n}{0}$ which has determinant $n$ leading to the bound on the width.
\end{proof}

Axis-aligned polygons have qualitative differences depending on the parity of $n$.  Let $A$ be an axis-aligned twisted $n$-gon with monodromy $\phi$.  Let $P,Q$ be the points of concurrency of its sides, say with the $\join{A_{2i-1}}{A_{2i}}$ passing through $P$.  Generically, $A_1,A_2,A_3,A_4$ are in general position and determine $P$ as
\begin{displaymath}
P = \meet{\join{A_1}{A_2}}{\join{A_3}{A_4}}.
\end{displaymath}
Recall $\phi(A_i) = A_{i+n}$ so
\begin{displaymath}
\phi(P) = \meet{\join{A_{n+1}}{A_{n+2}}}{\join{A_{n+3}}{A_{n+4}}}.
\end{displaymath}

First suppose $n$ is even.  Then $\join{A_{n+1}}{A_{n+2}}$ and $\join{A_{n+3}}{A_{n+4}}$ both pass through $P$ so $\phi(P)=P$.  A similar argument shows that $\phi(Q)=Q$.  By Theorem \ref{thmPenta}, there is some $k \leq n-1$ such that $T^k(A)$ is dual axis-aligned.  This differs from two previous results that consider only a subclass of the axis-aligned polygons.  The first result is that if $A$ is closed (i.e. $\phi$ is the identity map) then $T^k(A)$ becomes dual axis-aligned in at most $k=n/2-2$ steps \cite{S2, G}.  The second allows $\phi$ to be nontrivial provided that it fixes every point on the line $\join{P}{Q}$.  In this case $T^k(A)$ is axis-aligned for some $k \leq n/2-1$ steps \cite{G}.  

If $n$ is odd, then $\join{A_{n+1}}{A_{n+2}}$ and $\join{A_{n+3}}{A_{n+4}}$ both pass through $Q$ so by the above $\phi(P)=Q$.  By a similar argument $\phi(Q) = P$.  The fact that $\phi$ interchanges $P$ and $Q$ rules out the possibility that $A$ is closed.  Theorem \ref{thmPenta} is the first Devron type result for the pentagram map that includes polygons with an odd number of sides.

\subsection{Higher pentagram maps} \label{subsecHigher}
Gekhtman, Shapiro, Tabachnikov, and Vainshtein \cite{GSTV} consider a higher dimensional generalization of the pentagram map.  The objects of study are certain polygonal curves in $\mathbb{RP}^k$.  We generalize the notion of axis-aligned polygons to this setting and prove a Devron type result.

In $\mathbb{RP}^d$, a \emph{twisted $n$-gon} is a sequence of vertices $(A_i)_{i \in \mathbb{Z}}$ together with a projective transformation $\phi \in PGL_{d+1}(\mathbb{R})$ such that
\begin{displaymath}
A_{i+n} = \phi(A_i)
\end{displaymath}
for all $i$.  Assume for the definition that no three of $A_i, A_{i+1}, A_{i+d}, A_{i+d+1}$ are collinear for any $i$.  Such a polygon is called \emph{corrugated} if for all $i \in \mathbb{Z}$ the points
\begin{displaymath}
A_i, A_{i+1}, A_{i+d}, A_{i+d+1}
\end{displaymath}
are coplanar.  Let $\mathcal{P}_{d,n}$ be the space of twisted $n$-gons in $\mathbb{RP}^d$ and let $\mathcal{P}_{d,n}^0$ be the space of such $n$-gons that are corrugated.  Note in $\mathbb{RP}^2$ the coplanarity conditions are vacuous so $\mathcal{P}_{2,n}^0 = \mathcal{P}_{2,n}$ and is the same as the space of twisted $n$-gons considered in the previous subsection.  Also note that \cite{GSTV} uses the subscript $d+1$ instead of $d$ for the spaces so as to match the dimensions of the matrix $\phi$.

In \cite{GSTV}, Gekhtman, Shapiro, Tabachnikov, and Vainshtein define an operation on corrugated twisted polygons that they call higher pentagram maps.  Let $A \in \mathcal{P}_{d,n}^0$.  For each $i \in \mathbb{Z}$, the four points $A_i, A_{i+1}, A_{i+d}, A_{i+d+1}$ are coplanar, so the lines $\join{A_{i}}{A_{i+d}}$ and $\join{A_{i+1}}{A_{i+d+1}}$ intersect.  Define $T_d(A)$ to be the twisted polygon with vertices 
\begin{displaymath}
\ldots, \meet{\join{A_0}{A_d}}{\join{A_1}{A_{d+1}}}, \meet{\join{A_1}{A_{d+1}}}{\join{A_2}{A_{d+2}}}, \meet{\join{A_2}{A_{d+2}}}{\join{A_3}{A_{d+3}}}, \ldots.
\end{displaymath}
\begin{lem}[{\cite[Theorem 5.1]{GSTV}}]
If $A$ is a corrugated polygon then so is $T_d(A)$.
\end{lem}

Hence we have a sequence of maps $T_d : \mathcal{P}_{d,n}^0 \to \mathcal{P}_{d,n}^0$ for $d=2,3,4,\ldots$ and the map $T_2$ is just the ordinary pentagram map.  One can check that this operation is invertible and that $T_d^{-1}$ takes a polygon $B$ to one with vertices
\begin{displaymath}
\meet{\join{A_i}{A_{i+1}}}{\join{A_{i+d}}{A_{i+d+1}}}
\end{displaymath}
for $i \in \mathbb{Z}$.

We now generalize the content of the previous subsection.  Let $A \in \mathcal{P}_{d,n}^0$.  The vertices of the $T_d^k(A)$ can be placed on a lattice, but the nature of the lattice depends on the parity of $d$.  First suppose $d$ is even.  Let $P_{2i,0} = A_i$ for all $i \in \mathbb{Z}$.  Recursively define
\begin{equation}
P_{i,k} = \meet{\join{P_{i-d-1,k-1}}{P_{i+d-1,k-1}}}{\join{P_{i-d+1,k-1}}{P_{i+d+1,k-1}}}
\label{eqTd}
\end{equation}
for $k>0$ and
\begin{equation}
P_{i,k} = \meet{\join{P_{i-d-1,k+1}}{P_{i-d+1,k+1}}}{\join{P_{i+d-1,k+1}}{P_{i+d+1,k+1}}}
\label{eqTdInv}
\end{equation}
for $k < 0$.  The $P_{i,k}$ are defined for $i+k$ even.

If $d$ is odd, then let $P_{2i+1,0} = A_i$ for $i \in \mathbb{Z}$.  Define the $P_{i,k}$ for $k>0$ and $k<0$ using \eqref{eqTd} and \eqref{eqTdInv}.  The difference is that in this case $P_{i,k}$ is defined for $i$ odd.  Regardless of $d$, the $P_{i,k}$ for $k$ fixed give the vertices of $T_d^k(A)$.

If $i,d$ are such that $P_{i-d,k}$ and $P_{i+d,k}$ are defined, then two vertices of $T_d^{k+1}(A)$ lie on the $d$-diagonal $\join{P_{i-d,k}}{P_{i+d,k}}$, namely $P_{i-1,k+1}$ and $P_{i+1,k+1}$.  Define $y_{i,k} = y_{i,k}(A)$ by
\begin{equation}
y_{i,k} = -\chi(P_{i-d,k}, P_{i-1,k+1}, P_{i+1,k+1}, P_{i+d,k}).
\label{highery}
\end{equation}
If $d$ is even then $y_{i,k}$ is defined for $i+k$ even and if $d$ is odd then $y_{i,k}$ is defined for $i$ even.

\begin{prop}[\cite{GSTV}]
Let $A \in \mathcal{P}_{d,n}^0$, and $y_{i,k} = y_{i,k}(A)$.  Then
\begin{itemize}
\item $y_{i+2n,k} = y_{i,k}$ for all appropriate $i,k$.
\item The product of the distinct $y_{i,-1}^{-1}$ and the distinct $y_{i,0}$ equals 1.  More precisely, if $d$ is even then 
\begin{displaymath}
y_{1,-1}^{-1}y_{2,0}y_{3,-1}^{-1}y_{4,0}\cdots y_{2n,0} = 1
\end{displaymath}
and if $d$ is odd then
\begin{displaymath}
y_{2,-1}^{-1}y_{2,0}y_{4,-1}^{-1}y_{4,0} \cdots y_{2n,0} = 1
\end{displaymath}
\item The distinct $y_{i,-1}$ and $y_{i,0}$ determine $A$ up to projective equivalence and a rescaling operation.
\end{itemize}
\label{propHigheru}
\end{prop}

\begin{prop}[\cite{GSTV}]
Let $A \in \mathcal{P}_{d,n}^0$ and let $y_{i,k} = y_{i,k}(A)$ as above.  Then
\begin{displaymath}
y_{i,k-1}y_{i,k+1} = \frac{(1+y_{i-d-1,k})(1+y_{i+d+1,k})}{(1+y_{i-d+1,k}^{-1})(1+y_{i+d-1,k}^{-1})}.
\end{displaymath}
Put another way, if $Y_{i,j,k} = y_{(d-1)i+(d+1)j,k}$ for all $i+j+k$ even then the $Y_{i,j,k}$ satisfy \eqref{eqYRec}.
\end{prop}

Let $Y_{i,j,k} = y_{(d-1)i+(d+1)j,k}$.  Then
\begin{displaymath}
Y_{i+(d+1),j-(d-1),k} = Y_{i,j,k}
\end{displaymath}
and
\begin{displaymath}
Y_{i+n,j-n,k} = y_{(d-1)i + (d+1)j - 2n,k} = y_{(d-1)i + (d+1)j,k} = Y_{i,j,k}
\end{displaymath}
So for fixed $k$ the $Y_{i,j,k}$ are periodic with respect to 
\begin{displaymath}
\Lambda' = \lattice{d+1}{-(d-1)}{n}{-n}.
\end{displaymath}  
This is the companion of the lattice $\Lambda = \lattice{d}{1}{n}{0}$.  

Letting $u_{i,j} = Y_{i,j,0}$ for $i+j$ even and $u_{i,j} = Y_{i,j,-1}^{-1}$ for $i+j$ odd we have $(u,\sigma) \in X_{\Lambda}$ for $\Lambda = \lattice{d}{1}{n}{0}$.  There are $2\det(\Lambda) = 2n$ distinct $u_{i,j}$, and by the second part of Proposition \ref{propHigheru} their product is 1.  Therefore $(u,\sigma) \in \overline{X}_{\Lambda}$.  In short, the higher pentagram map $T_d$ is given in certain coordinates by the periodic $Y$-system map 
\begin{displaymath}
G: \overline{X}_{\Lambda} \to \overline{X}_{\Lambda}
\end{displaymath}
for $\Lambda = \lattice{d}{1}{n}{0}$.

We immediately get the Devron property for higher pentagram maps.  Again, the classes of singular polygons can be described geometrically.  Let $A \in \mathcal{P}_{d,n}$.  Say $A$ is \emph{axis-aligned} if there are points $Q_1,Q_2,\ldots, Q_d$ such that the side $\join{A_{i+jd}}{A_{i+jd+1}}$ goes through $Q_i$ for all $i=1,2,\ldots, d$, $j \in \mathbb{Z}$.  We can choose the points of concurrency $Q_i$ at infinity in such a way that the sides of $A$ are parallel to the $d$ axes in $\mathbb{R}^d \subseteq \mathbb{RP}^d$.  Suppose $A$ is axis-aligned.  Then $\join{A_i}{A_{i+1}}$ and $\join{A_{i+d}}{A_{i+d+1}}$ intersect at one of the $Q_j$ so in particular these four points are coplanar.  Therefore $A$ is corrugated.

Say that $A$ is \emph{dual axis-aligned} if there are lines $L_1,L_2,\ldots, L_d$, all passing through a common point, such that $L_i$ contains $A_{i+jd}$ for all $i=1,2,\ldots, d$, $j \in \mathbb{Z}$.  If $A$ is dual axis-aligned then $\join{A_i}{A_{i+d}}$ equals one of the $L_j$, as does $\join{A_{i+1}}{A_{i+d+1}}$.  The $L_j$ all intersect so these two $d$-diagonals intersect.  Again $A$ must be corrugated.

\begin{rmk}
The properties defined above are projectively dual to each other, but for $d>2$ the given definitions do not make this fact clear.  More precisely, $A$ is dual axis-aligned if and only if the dual polygon is axis-aligned.  We briefly explain one direction of this equivalence.  Suppose $A$ is dual axis-aligned.  Let $L_1,L_2,\ldots L_d$ the lines containing each of its vertices and let $P$ be the point where the $L_i$ all intersect.  For $i=1,2,\ldots, d$, let $H_i$ be the hyperplane containing $L_1, L_2, \ldots, L_{i-1},L_{i+1},\ldots, L_d$.  It is helpful to think of $P$ as the origin, the $L_i$ as the coordinate axes, and the $H_i$ as the coordinate hyperplanes.  The sides of the dual of $A$ correspond to the codimension two subspaces spanned by $d-1$ consecutive vertices of $A$, say $A_{i+1},A_{i+2},\ldots, A_{i+d-1}$.  Each vertex is on one of the $L_j$, so each such subspace is contained in one of the $H_j$ (specifically the one with $j \equiv i \pmod{d}$).  This is dual to the notion of axis-aligned, in which each side passes through one of $d$ points.
\end{rmk}

\begin{prop}
Let $A \in \mathcal{P}_{d,n}^0$ and let $y_{i,k} = y_{i,k}(A)$.  
\begin{enumerate}
\item $A$ is axis-aligned if and only if the $y_{i,-1}$ all equal -1.
\item $A$ is dual axis-aligned if and only if the $y_{i,0}$ all equal -1.
\end{enumerate}
\end{prop}

\begin{proof}
By \eqref{highery}, the $y_{i,k}$ are negated cross ratios.  In general, $\chi(a,b,c,d)=1$ if and only if $a=d$ or $b=c$.
\begin{enumerate}
\item We have
\begin{displaymath}
y_{i,-1} = -\chi(B, A_j, A_{j+1}, C)
\end{displaymath}
where $B$ and $C$ are two vertices of $T_d^{-1}(A)$ a distance $d$ apart (here $i=2j+1$ if $d$ is even and $i=2j+2$ if $d$ is odd).  The definition of twisted polygon does not allow $A_j=A_{j+1}$.  Hence the $y_{i,-1}$ are all -1 if and only if the vertices  of $T_d^{-1}(A)$ are periodic with period $d$.  Each side of $A$ passes through one of these $d$ points, for instance the side $\join{A_j}{A_{j+1}}$ above passes through $B=C$.  Therefore $A$ is axis-aligned.  On the other hand, $B$ and $C$ can be computed as the intersection of $\join{A_j}{A_{j+1}}$ with the sides of $A$ that occur $d$ earlier and $d$ later respectively.  If $A$ is axis-aligned then these three sides pass through a common point so $B=C$.
\item For each $j \equiv d \pmod{2}$
\begin{displaymath}
y_{i,0} = -\chi(A_{(j-d)/2}, B, C, A_{(j+d)/2})
\end{displaymath}
where $B$ and $C$ are two consecutive vertices of $T_d(A)$ (here $i=j$ if $d$ is even and $i=j+1$ if $d$ is odd).  The definition of twisted polygon does not allow $A_{(j-d)/2} = A_{(j+d)/2}$ so $y_{i,0}=-1$ if and only if $B = C$.  This holds for all $i$ if and only if all vertices of $T_d(A)$ are equal to some common point $Q$.  By definition of $T_d$, each $d$-diagonal of $A$ goes through $Q$ in this case.  For each $i \in \mathbb{Z}$, we have that $\join{A_{i-d}}{A_i}$ and $\join{A_i}{A_{i+d}}$ both contain $Q$ and $A_i$, so they must be equal.  Hence $A_{i-d}, A_i, A_{i+d}$ are collinear.  It follows that all the $A_j$ with $j \equiv d \pmod{i}$ lie on some common line that passes through $Q$.  Therefore $A$ is dual axis aligned.  
\end{enumerate}
\end{proof}

The notions of axis-aligned and dual axis-aligned translate to the singularities of $G^{-1}$ and $G$ respectively that are considered in Theorem \ref{thmYDevron}.  The following then arises by specializing said theorem to the lattice $\lattice{d}{1}{n}{0}$.

\begin{thm}
Let 
\begin{align*}
U &= \{A \in \mathcal{P}_{d,n}^0 : A \textrm{ axis-aligned}\} \\
V &= \{A \in \mathcal{P}_{d,n}^0 : A \textrm{ dual axis-aligned}\} 
\end{align*}
Then $(U,V)$ is a Devron pair for the higher pentagram map $T_d$.  The width of the pair is at most $n-1$.
\end{thm}

\subsection{The lower pentagram map}
As explained above, the $Y$-system \eqref{eqYRec} with periodicity lattice $\lattice{d}{1}{n}{0}$ for $d \geq 2$ encodes the higher pentagram map in dimension $d$.  The case $d=2$ corresponds to the original pentagram map.  Gekhtman, Shapiro, Tabachnikov, and Vainshtein also give a geometric interpretation of the system with $\Lambda = \lattice{1}{1}{n}{0}$  which fittingly can be seen as a one dimensional, or lower, pentagram map.

Let $\mathcal{P}_{1,n}$ be the space of twisted polygons in $\mathbb{RP}^1$ (defined analogously to the previous notions of twisted polygon).  Let $X$ be the space of pairs $(A,B)$ of twisted polygons on the line with a common monodromy.  The \emph{lower pentagram map} $T_1: X \to X$ is defined by $T(A,B) = (B,C)$ where each $C_i$ is constructed from $A_i, B_{i-1}, B_i, B_{i+1}$ by one of the following equivalent procedures
\begin{itemize}
\item $C_i$ is the image of $A_i$ under the unique projective transformation that fixes $B_i$ and interchanges $A_{i-1}$ and $A_{i+1}$.
\item $C_i$ is the unique point for which $[B_{i-1},B_i,A_i,B_{i+1},A_i,C_i] = -1$ where
\begin{displaymath}
[a,b,c,d,e,f] = \frac{(a-b)(c-d)(e-f)}{(b-c)(d-e)(f-a)}
\end{displaymath}
\end{itemize}
This operation is invertible.  In fact for fixed $B$, we can recover $A$ from $C$ using the same procedure that produced $C$ from $A$.

We move through the discussion of the Devron property for this system quickly, as it is quite similar to the situation for the higher pentagram map.  Fix $(A,B) \in X$.  Let $P_{2i+1,-1} = A_i$ and $P_{2i+1,0}=B_i$ for $i \in \mathbb{Z}$.  Recursively define $P_{i,k}$ for $i$ odd an $k>0$ or $k<-1$ in such a manner that
\begin{displaymath}
[P_{i-2,k}, P_{i,k}, P_{i,k-1}, P_{i+2,k}, P_{i,k}, P_{i,k+1}] = -1
\end{displaymath}
always holds.  The $P_{i,k-1}$ and $P_{i,k}$ are the points that make up $T^k(A,B)$.  For $i,k \in \mathbb{Z}$ with $i$ even, let
\begin{displaymath}
y_{i,k} = y_{i,k}(A,B) = -\chi(P_{i-1,k}, P_{i-1,k+1}, P_{i+1,k+1}, P_{i+1,k}).
\end{displaymath}

\begin{prop}[\cite{GSTV}]
Let $(A,B) \in X$ and let $y_{i,k} = y_{i,k}(A,B)$ as above.
\begin{itemize}
\item $y_{i+2n,k} = y_{i,k}$ for all $i,k$ with $k$ odd.
\item $y_{1,-1}^{-1}y_{1,0}y_{3,-1}^{-1}y_{3,0} \cdots y_{2n-1,0} = 1$
\item The distinct $y_{i,-1}$ and $y_{i,0}$ determine $A$ up to projective equivalence and a rescaling operation.
\end{itemize}
\end{prop}

\begin{prop}[\cite{GSTV}]
Let $(A,B) \in X$ and let $y_{i,k} = y_{i,k}(A,B)$ as above.  Then
\begin{displaymath}
y_{i,k-1}y_{i,k+1} = \frac{(1+y_{i-2,k})(1+y_{i+2,k})}{(1+y_{i,k}^{-1})^2}
\end{displaymath}
for all $i,k \in \mathbb{Z}$ with $i$ even.  Put another way, if $Y_{i,j,k} = y_{2j,k}$ for $i+j+k$ even then the $Y_{i,j,k}$ satisfy \eqref{eqYRec}.
\end{prop}

For fixed $k$, the $Y_{i,j,k}$ defined above are clearly periodic with respect to
\begin{displaymath}
\Lambda' = \lattice{2}{0}{n}{-n}
\end{displaymath}
which is the companion to $\Lambda = \lattice{1}{1}{n}{0}$.  Using similar reasoning to before, the lower pentagram map $T_1$ is described by the $Y$-system map
\begin{displaymath}
G: \overline{X}_{\Lambda} \to \overline{X}_{\Lambda}
\end{displaymath}
for this $\Lambda$.

\begin{thm}
Let $X$ and $T_1 : X \to X$ be as before.  Define $U,V \subseteq X$ by
\begin{align*}
U &= \{(A,B) : \textrm{ the $A_i$ are constant}\} \\
V &= \{(A,B) : \textrm{ the $B_i$ are constant}\}
\end{align*}
Then $U,V$ is a Devron pair for $T_1$.  The width of the pair is at most $n$.
\end{thm}

\begin{proof}
Let $(A,B) \in U$ and let $y_{i,k} = y_{i,k}(A)$, $Y_{i,j,k} = y_{2j,k}$ as before.  For all $i$ 
\begin{displaymath}
y_{2i,-1} = -\chi(A_{i-1},B_{i-1},B_i,A_i) = -1
\end{displaymath}
since $A_{i-1} = A_i$.  So $Y_{i,j,-1} = -1$ for all $i+j$ odd.  By Theorem \ref{thmYDevron}, there is an $N \leq n-1$ such that 
\begin{itemize}
\item the $Y_{i,j,N} = -1$ 
\item for $1 \leq k < N$ and generic $A$ the $Y_{i,j,k} \neq -1$.
\end{itemize}
Let $(A',B') = T_1^{N+1}(A,B)$.  Then
\begin{displaymath}
-1 = Y_{i,j,N} = y_{2j,N} = -\chi(A'_{j-1}, B'_{j-1}, B'_j, A'_j)
\end{displaymath}
for all $j \in \mathbb{Z}$, so $A'_{j-1}= A'_j$ or $B'_{j-1}=B'_j$.  Having $A'_{j-1}=A'_{j}$ would cause $Y_{i,j,N-1}=-1$ which is generically false.  So we have $B'_{j-1} = B'_j$ for all $j$.  Therefore $(A',B') \in V$ as desired.
\end{proof}

\section{Polygon recutting} \label{secCut}
Polygon recutting is an operation on polygons in the plane first studied by Adler \cite{A}.  A triangle is cut from the rest of the polygon by cutting along a shortest diagonal.  The triangle is then reattached along the same diagonal but with the opposite orientation.  Figure \ref{figRecut} shows an instance of this procedure.  If the cut is made between vertices $i-1$ and $i+1$, then only vertex $i$ ends up changing position.  As such, call this operation recutting at vertex $i$.

For the purpose of this section, a polygon is simply a cyclically ordered $n$-tuple of points (its vertices) in $\mathbb{R}^2$.  Let $X_n$ be the space of $n$-gons, so $X_n \cong \mathbb{R}^{2n}$.  For $i=1,\ldots, n$ let $s_i: X_n \to X_n$ denote the operation of recutting at vertex $i$.  Each $s_i$ is a rational map, which is defined precisely when vertices $i-1$ and $i+1$ are distinct.  Moreover, it is clear that $s_i$ is an involution, so in particular it is birational.

\begin{lem}[\cite{A}]
The functions $s_i$ on $X_n$ satisfy the relations of the affine symmetric group, namely
\begin{itemize}
\item $s_i^2 = 1$ for $i=1,2,\ldots, n$
\item $s_is_j = s_js_i$ if $i-j \equiv 2,3,\ldots, n-2 \pmod{n}$
\item $s_is_js_i = s_js_is_j$ if $i-j \equiv \pm 1 \pmod{n}$.
\end{itemize}
\label{lemBraid}
\end{lem}

\begin{lem}
Let $A,B,C,D \in \mathbb{R}^2$ be distinct points.  Then recutting at $B$ in $\triangle ABD$ yields $\triangle ACD$ if and only if recutting at $A$ in $\triangle BAC$ produces $\triangle BDC$.
\label{lemConsistency}
\end{lem}

\begin{proof}
Both conditions are equivalent to $A,B,C,D$ being the vertices of an isosceles trapezoid with sides $\overline{BC}$ and $\overline{AD}$ parallel (see Figure \ref{figConsistency}).
\end{proof}

\begin{figure}
\begin{pspicture}(10,3)
\pnode(1,1){A}
\pnode(2,2){B}
\pnode(3,2){C}
\pnode(4,1){D}
\pspolygon(A)(B)(D)(C)
\psline[linestyle=dashed](A)(D)
\uput[dl](A){$A$}
\uput[ul](B){$B$}
\uput[ur](C){$C$}
\uput[dr](D){$D$}


\rput(5,0){
\pnode(1,1){A2}
\pnode(2,2){B2}
\pnode(3,2){C2}
\pnode(4,1){D2}
\pspolygon(B2)(A2)(C2)(D2)
\psline[linestyle=dashed](B2)(C2)
\uput[dl](A2){$A$}
\uput[ul](B2){$B$}
\uput[ur](C2){$C$}
\uput[dr](D2){$D$}
}
\end{pspicture}
\caption{Consistency for polygon recutting}
\label{figConsistency}
\end{figure}
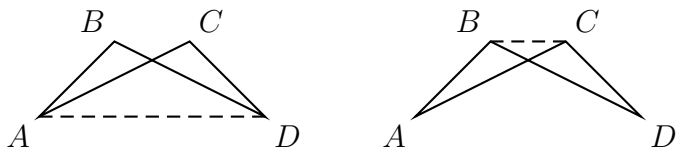

It is possible to compose the various $s_i$ in any order, but we will consider the so-called bipartite dynamics.  Assume an even number of sides and perform all the $s_i$ with $i$ odd, then all the $s_i$ with $i$ even, and continue in this manner.  For a fixed parity of $i$, the $s_i$ commute, so there is no need to specify an order.  What we obtain is a map 
\begin{displaymath}
F:X_{2n} \times \{0,1\} \to X_{2n} \times \{0,1\}
\end{displaymath}
defined by $F(A,\sigma) = (A', 1 - \sigma)$ where
\begin{displaymath}
A' = \begin{cases}
s_{2n} \circ \cdots \circ s_4 \circ s_2(A), & \sigma = 0 \\
s_{2n-1} \circ \cdots \circ s_3 \circ s_1(A), & \sigma = 1 \\
\end{cases}
\end{displaymath}

Lastly, define $W_0, W_1 \subseteq X_{2n}$ by
\begin{align*}
W_0 &= \{A \in X_{2n} : A_2 = A_4 = \ldots = A_{2n}\} \\
W_1 &= \{A \in X_{2n} : A_1 = A_3 = \ldots = A_{2n-1}\} 
\end{align*}
and define $U,V \subseteq X_{2n} \times \{0,1\}$ by
\begin{align*}
U &= (W_0 \times \{0\}) \cup (W_1 \times \{1\}) \\
V &= (W_0 \times \{1\}) \cup (W_1 \times \{0\}) 
\end{align*}

\begin{thm}
Let $F:X_{2n}\times \{0,1\} \to X_{2n}\times \{0,1\}$ and $U,V$ be as above.  Then $U,V$ is a Devron pair of width $n-1$.
\label{thmRecut}
\end{thm}

\begin{lem}
Define $G: X_{n+1} \to X_{n+1}$ by 
\begin{displaymath}
G = s_n \circ s_{n-1} \circ \cdots \circ s_3 \circ s_2.
\end{displaymath}
Then $G^n = 1$.
\label{lemOrder}
\end{lem}

\begin{proof}
We are working with $(n+1)$-gons, so vertices 1 and $n$ are not adjacent.  Therefore $s_1$ and $s_n$ commute, which along with the rest of Lemma \ref{lemBraid} implies that $s_1,\ldots, s_n$ satisfy the relations of the symmetric group $S_{n+1}$.  Viewed as a permutation, $s_ns_{n-1}\cdots s_3s_2$ is an $n$-cycle, so $(s_ns_{n-1}\cdots s_3s_2)^n = 1$ as desired.
\end{proof}

Let 
\begin{displaymath}
S = \{(i,j) \in \mathbb{Z}^2 : i+j \textrm{ even}, 1 \leq i \leq n+1\} 
\end{displaymath}
Let $A \in X_{n+1}$.  Define a grid of points $C_{i,j}$ indexed by $(i,j) \in S$ as follows.  Let $C_{i,i} = A_i$ for $i=1,2,\ldots, n+1$.  More generally, for $i=1,2,\ldots, n+1$ and $k \in \mathbb{Z}$, let $C_{i,i+2k}$ be the $i$th vertex of $G^k(A)$ for $G$ as in Lemma \ref{lemOrder}.  Note that $G$ has no effect on vertices $1$ and $n+1$, so $C_{1,j} = A_1$ and $C_{n+1,j} = A_{n+1}$ for all $j$ of the appropriate parities.  By Lemma \ref{lemOrder}, $G^{k+n}(A) = G^k(A)$ for all $k$.  Therefore $G_{i,j+2n} = G_{i,j}$ for all $(i,j) \in S$.  

Define 
\begin{displaymath}
B_j = \begin{cases}
C_{1,j}, & j \textrm{ odd} \\
C_{2,j}, & j \textrm{ even} \\
\end{cases}
\end{displaymath}
By the above, the $B_i$ are $2n$-periodic and $B_1=B_3 = \ldots = B_{2n-1}$.  Therefore $B \in W_1$.  So the procedure producing $B$ from $A$ is a rational map $\phi: X_{n+1} \to W_0$.

Say that $A \in X_{n+1}$ is \emph{jittery} if the $n$ side lengths $d(A_1,A_2), \ldots, d(A_n,A_{n+1})$ are all nonzero and distinct.  

\begin{lem}
Suppose that $A \in X_{n+1}$ is jittery.
\begin{enumerate}
\item The maps $G^i$ with $i=1,2,\ldots, n$ are all defined at $A$.  Moreover, each successive iteration moves each vertex other than $A_1$ and $A_{n+1}$.
\item Let $C = C(A)$.  For each $i,j \in \mathbb{Z}$ with $i+j$ odd and $2 \leq i \leq n$, the points $C_{i-1,j}$, $C_{i,j-1}$, $C_{i,j+1}$ and $C_{i+1,j}$ are the vertices of an isosceles trapezoid with sides $\overline{C_{i-1,j} C_{i+1,j}}$ and $\overline{C_{i,j-1}C_{i,j+1}}$ parallel.
\item Let $B = \phi(A)$.  For each $k=0,1,\ldots n-1$, the vertices of $F^k(B,1)$ are
\begin{displaymath}
\ldots, C_{k+1,1}, C_{k+2,2}, C_{k+1,3}, C_{k+2,4}, \ldots
\end{displaymath}
if $k$ is even and
\begin{displaymath}
\ldots, C_{k+2,1}, C_{k+1,2}, C_{k+2,3}, C_{k+1,4}, \ldots
\end{displaymath}
if $k$ is odd.  In particular, $A_1 = B_1$ and $A_i=C_{i,i}$ is the $i$th vertex of $F^{i-2}(B,1)$ for $i=2,3,\ldots, n+1$.
\item $F^{n-1}(A,1) \in V$.
\end{enumerate}
\label{lemJitter}
\end{lem}

\begin{proof} \ \newline
\begin{enumerate}
\item In general, $s_i$ fixes all of the side lengths of the polygon except for $d(A_{i-1},A_i)$ and $d(A_i, A_{i+1})$ which get interchanged.  Moreover, if these two side lengths are distinct then it is impossible to have $A_{i-1}=A_{i+1}$, so $s_i$ must be nonsingular.  Lastly, in the case of distinct side lengths, $s_i$ acts nontrivially on $A_i$.  

Since $A$ is jittery, each of $s_2,s_3,\ldots s_{n-2}$ is defined at $A$ and produces another jittery polygon.  As such, any word in these $s_i$, including $G$ and its iterates, is defined at $A$.  In such a word, each individual $s_i$ will act nontrivially on vertex $i$.  Since $G$ contains each of $s_2,\ldots, s_{n-2}$ exactly once, it will act nontrivially on each of the corresponding vertices.   

\item Applying $G$ creates a translational symmetry in the $j$ direction.  As such, we can assume $j=i+1$.  The points in question are $C_{i-1,i+1}$, $C_{i,i}=A_i$, $C_{i,i+2}$ and $C_{i+1,i+1}=A_{i+1}$, i.e., the primitive diamonds in the figure
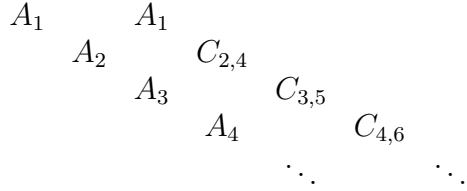
\begin{displaymath}
\begin{array}{ccccccc}
A_1 & & A_1 \\
& A_2 & & C_{2,4} \\
& & A_3 & & C_{3,5} \\
& & & A_4 & & C_{4,6} \\
& & & & \ddots & & \ddots \\
\end{array}
\end{displaymath}
(recall $C_{1,j} = A_1$ for all $j$).  By definition the $C_{i,i+2}$ are the vertices of $G(A) = s_{n}\cdots s_3s_2(A)$.  Each $s_i$ only modifies vertex $i$, so $s_{i-1} \cdots s_3s_2(A)$ has vertices
\begin{displaymath}
A_1, C_{2,4}, C_{3,5}, \ldots C_{i-1,i+1}, A_{i}, A_{i+1}, \ldots, A_{n+1}
\end{displaymath}
at which point $s_i$ changes $A_i$ to $C_{i,i+2}$.  The result follows from the proof of Lemma \ref{lemConsistency}.

\item Do induction on $k$.  The base case $k=0$ is simply the definition of $\phi$.  The inductive case follows from the previous and the proof of Lemma \ref{lemConsistency}

\item Suppose $n$ is odd.  Then by the previous part $F^{n-1}(B,1)$ has vertices 
\begin{displaymath}
\ldots C_{n,1}, C_{n+1,2}, C_{n,3}, C_{n+1,4}, \ldots.
\end{displaymath}
However, $C_{n+1,2j} = A_{n+1}$ for all $j$, so this polygon is in $W_0$ and $F^{n-1}(B,1) \in W_0 \times \{1\} \subseteq V$.  Similarly, if $n$ is even then $F^{n-1}(B,1) \in W_1 \times \{0\} \subseteq V$.
\end{enumerate}
\end{proof}

\begin{proof}[Proof of Theorem \ref{thmRecut}]
It is clear that $F^{-1}$ is singular on $U$ and $F$ is singular on $V$.

Let $(B,\sigma) \in U$.  Assume without loss of generality that $\sigma = 1$.  Suppose first that $B=\phi(A)$ for some jittery $A \in X_{n+1}$.  By part 4 of Lemma \ref{lemJitter}, we have $F^{n-1}(B,\sigma) \in V$.  By part 1 of the same lemma, $A$ gives rise to a grid $C_{i,j}$ of points with $C_{i-1,j} \neq C_{i+1,j}$ for all $i+j$ odd, $2 \leq j \leq n$.  Also, by part 3 of the lemma, the vertices of $F^k(B)$ are precisely the $C_{k+1,j}$ and $C_{k+2,j}$.  When $F$ is iterated starting from $(B,\sigma)$ each shortest diagonal is defined so each individual $s_i$ can be applied and is invertible.  Therefore $F^{n-1}$ is defined at $(B,\sigma)$ and $F^{-(n-1)}$ is defined at $F^{n-1}(B,\sigma)$.

We get that the conditions needed for the Devron property hold provided that the input $B$ is of the form $B=\phi(A)$ for some $A$ jittery.  Part 3 of Lemma \ref{lemJitter} implies that $\phi$ is injective when restricted to the space of jittery $n+1$-gons.  The set of such is an open subset of $X_{n+1} \cong \mathbb{R}^{2(n+1)}$, so its image contains an open ball in $W_1 \cong \mathbb{R}^{2(n+1)}$.  It follows that the needed results hold generically on $W_1$ as desired.
\end{proof}

\section{Generalized discrete Toda system} \label{secToda}
We now consider a variant of the generalized discrete Toda system.  Define a map $\mu:\mathbb{C}^{2 \times 2} \to \mathbb{C}^{2 \times 2}$ by
\begin{displaymath}
\mu\left(\left[\begin{array}{cc}
x_1 & y_1 \\
x_2 & y_2 \\
\end{array}
\right]\right)
= \left[\begin{array}{cc}
y_1\frac{x_2+y_2}{x_1+y_1} & x_1\frac{x_2+y_2}{x_1+y_1} \\
y_2\frac{x_1+y_1}{x_2+y_2} & x_2\frac{x_1+y_1}{x_2+y_2} \\
\end{array}\right].
\end{displaymath}
One thinks of this procedure as pushing the $x$ and $y$ columns past each other and accumulating extra factors in the process.  Define maps $s_j: \mathbb{C}^{2 \times n} \to \mathbb{C}^{2 \times n}$ for $j=1,2,\ldots, n$ where $s_j$ applies $\mu$ to columns $j$ and $j+1$ and leaves the rest of the columns fixed.  Column indices are considered modulo $n$, so $s_n$ moves columns $n$ and $1$ past each other.

There is a more general notion of pushing two length $m$ vectors past each other.  This procedure forms the basis for the discrete Toda lattice.  The generalized Toda lattice considered by Iwao in \cite{I} corresponds roughly to the map $s_n\circ\ldots \circ s_1$.  The difference is that the map $s_n$ includes a cyclic shift of the entires of columns $1$ and $n$ in that model.  We will look at the system defined by applying the $s_j$ in a bipartite manner.  Moreover, we assume $m=2$ as our main result does not generalize in a straightforward manner to larger $m$.

Formally, define $F:\mathbb{C}^{2 \times 2n} \times \{0,1\} \to \mathbb{C}^{2 \times 2n} \times \{0,1\}$ by $F(x,\sigma) = (x',1-\sigma)$ where
\begin{displaymath}
x' = \begin{cases}
s_{2n} \circ \cdots \circ s_4 \circ s_2(x), & \sigma = 0 \\
s_{2n-1} \circ \cdots \circ s_3 \circ s_1(x), & \sigma = 1 \\
\end{cases}
\end{displaymath}
Note that $s_j$ is singular when $x_{1,j}=-x_{1,j+1}$ and $x_{2,j}=-x_{2,j+1}$.  As such, let
\begin{align*}
W_0 &= \left\{\left[\begin{array}{ccccccc}
-x_{1,2n} & x_{1,2} & -x_{1,2} & x_{1,4} & -x_{1,4} & \cdots & x_{1,2n} \\
-x_{2,2n} & x_{2,2} & -x_{2,2} & x_{2,4} & -x_{2,4} & \cdots & x_{2,2n} \\
\end{array}
\right]\right\} \\
W_1 &= \left\{\left[\begin{array}{ccccccc}
x_{1,1} & -x_{1,1} & x_{1,3} & -x_{1,3} & \cdots & x_{1,2n-1} & -x_{1,2n-1} \\
x_{2,1} & -x_{2,1} & x_{2,3} & -x_{2,3} & \cdots & x_{2,2n-1} & -x_{2,2n-1} \\
\end{array}
\right]\right\} 
\end{align*} 
and define $U,V \subseteq \mathbb{C}^{2 \times 2n} \times \{0,1\}$ by
\begin{align*}
U &= (W_0 \times \{1\}) \cup (W_1 \times \{0\}) \\
V &= (W_0 \times \{0\}) \cup (W_1 \times \{1\}) 
\end{align*}

\begin{thm} \label{thmToda}
Let $F$, $U$, and $V$ be as above.  Then $U,V$ is a Devron pair of width $n-1$.
\end{thm}

It turns out the polygon recutting can be described as a specialization of the map $\mu$ to the case where the two entries of each column are complex conjugates.  

\begin{prop} \label{propTodaCut}
Let $A, B, C \in \mathbb{R}^2 \cong \mathbb{C}$ and let $B'$ be the result of polygon recutting at $B$.  Define $w, z, w', z' \in \mathbb{C}$ by $w = B-A$, $z= C-B$, $w' = C-B'$, and $z' = B'-A$.  Then
\begin{displaymath}
\mu\left(\left[\begin{array}{cc}
w & z \\
\overline{w} & \overline{z} \\
\end{array}
\right]\right)
= \left[\begin{array}{cc}
\overline{z'} & \overline{w'} \\
z' & w' \\
\end{array}
\right]
\end{displaymath}
\end{prop}

\begin{proof}
By the definition of $\mu$, we need to show
\begin{align*}
z' &= \overline{z}\frac{w+z}{\overline{w}+\overline{z}} \\
w' &= \overline{w}\frac{w+z}{\overline{w}+\overline{z}} 
\end{align*}
Taking these as the definition of $z'$ and $w'$, it is clear that $|z'|=|z|$, $|w'| = |w|$, and $z'+w' = w+z$.  In general, there are two pairs $(z',w')$ with these properties, namely $(z,w)$ and $(B'-A,C-B')$ (see Figure \ref{figTodaCut}).  A direct calculation shows that $z' \neq z$ except when $w$ and $z$ are parallel in which case the two solutions are equal.  Therefore $(z',w') = (B'-A, C-B')$.
\end{proof}

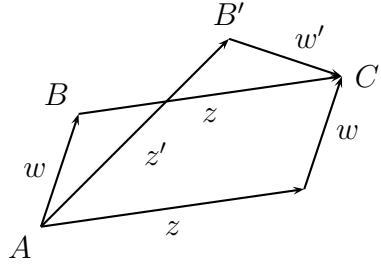
\begin{figure}
\begin{pspicture}(4,4)
\pnode(1,1){A}
\pnode(1.5,2.5){B}
\pnode(3.5,3.5){B1}
\pnode(5,3){C}
\pnode(4.5,1.5){D}

\psline[arrows=->](A)(B)
\psline[arrows=->](B)(C)
\psline[arrows=->](A)(B1)
\psline[arrows=->](B1)(C)
\psline[arrows=->](A)(D)
\psline[arrows=->](D)(C)
\uput[dl](A){$A$}
\uput[ul](B){$B$}
\uput[u](B1){$B'$}
\uput[r](C){$C$}
\uput[l](1.25,1.75){$w$}
\uput[d](3.25,2.75){$z$}
\uput[dr](2.25,2.25){$z'$}
\uput[ur](4.25,3.25){$w'$}
\uput[d](2.75,1.25){$z$}
\uput[r](4.75,2.25){$w$}
\end{pspicture}
\caption{An illustration of the proof of Proposition \ref{propTodaCut}}
\label{figTodaCut}
\end{figure}

As a result, Theorem \ref{thmToda} can be thought of as a generalization of Theorem \ref{thmRecut}, and the proof involves the same basic steps.  The translation from polygons to matrices is complicated by the fact that $\mu$ moves the relevant vectors from the top row to the bottom row.  Define an auxiliary map
\begin{displaymath}
\nu\left(\left[\begin{array}{cc}
x_1 & y_1 \\
x_2 & y_2 \\
\end{array}
\right]\right)
= \left[\begin{array}{cc}
-y_1\frac{x_1-y_2}{x_2-y_1} & -x_1\frac{x_2-y_1}{x_1-y_2} \\
-y_2\frac{x_2-y_1}{x_1-y_2} & -x_2\frac{x_1-y_2}{x_2-y_1} \\
\end{array}\right].
\end{displaymath}
The following is then analogous to Lemma \ref{lemConsistency}.

\begin{lem} \label{lemTodaConsistency}
The conditions 
\begin{displaymath}
\mu\left(\left[\begin{array}{cc}
x_1 & y_1 \\
x_2 & y_2 \\
\end{array}
\right]\right)
= \left[\begin{array}{cc}
y_1' & x_1' \\
y_2' & x_2' \\
\end{array}\right]
\end{displaymath}
and
\begin{displaymath}
\nu\left(\left[\begin{array}{cc}
x_1 & y_1' \\
x_2 & y_2' \\
\end{array}
\right]\right)
= \left[\begin{array}{cc}
-y_1 & -x_1' \\
-y_2 & -x_2' \\
\end{array}\right]
\end{displaymath}
are equivalent.
\end{lem}
\begin{proof}
This is a straightforward verification.  For example, suppose the first condition holds and compare the top left entries of both sides of the second condition.  On the right is $-y_1$ and on the left, by definition of $\nu$, is
\begin{displaymath}
-y_1'\frac{x_1-y_2'}{x_2-y_1'}
\end{displaymath}
Now 
\begin{displaymath}
x_1-y_2' = x_1-y_2\frac{x_1+y_1}{x_2+y_2} = \frac{x_1x_2-y_1y_2}{x_2+y_2}
\end{displaymath}
and
\begin{displaymath}
x_2-y_1' = x_2 - y_1\frac{x_2+y_2}{x_1+y_1} = \frac{x_1x_2-y_1y_2}{x_1+y_1}
\end{displaymath}
Therefore
\begin{displaymath}
-y_1'\frac{x_1-y_2'}{x_2-y_1'} = -y_1\frac{x_2+y_2}{x_1+y_1}\frac{x_1+y_1}{x_2+y_2} = -y_1
\end{displaymath}
as desired.
\end{proof}

Define $t_j: \mathbb{C}^{2 \times n} \to \mathbb{C}^{2 \times n}$ for $j=1,2,\ldots, n$ as the map that applies $\nu$ to columns $j$ and $j+1$ and leaves the rest fixed.  The $t_j$ satisfy the relations of the affine symmetric group, a fact that can be verified directly or deduced from Theorem 6.3 of \cite{LP}.

\begin{lem}
The functions $t_i$ satisfy the relations of the affine symmetric group, namely
\begin{itemize}
\item $t_i^2 = 1$ for $i=1,2,\ldots, n$
\item $t_it_j = t_jt_i$ if $i-j \equiv 2,3,\ldots, n-2 \pmod{n}$
\item $t_it_jt_i = t_jt_it_j$ if $i-j \equiv \pm 1 \pmod{n}$.
\end{itemize}
\end{lem}

\begin{proof}[Proof of Theorem \ref{thmToda}]
Let $y \in \mathbb{C}^{2\times n}$ and let $y_1,y_2,\ldots, y_n$ denote the columns of $y$.  Consider just the maps $t_1,\ldots, t_{n-1}$ which satisfy the relations of the symmetric group $S_n$.  Let $G_0$ be the composition of such $t_j$ with $j$ even, and $G_1$ the same with $j$ odd.  For example, if $n$ is even then $G_0 = t_{n-2} \circ \ldots \circ t_4 \circ t_2$.  Consider the sequence of matrices 
\begin{displaymath}
y, G_0(y), G_1(G_0(y)), G_0(G_1(G_0(y))), \ldots
\end{displaymath}
and let $\tilde{w}_{j,k}$ be the $j$th column of the $k$th matrix in this sequence, starting with $\tilde{w}_{j,0} = y_j$.  One can check that $G_1 \circ G_0$ has order $n$, so $\tilde{w}_{j,k+2n} = \tilde{w}_{j,k}$ for all $j$ and $k$.  Let $w_{j,k} = (-1)^k\tilde{w}_{j,k}$.  Finally, let $x \in \mathbb{C}^{2\times 2n}$ be the matrix with columns $w_{1,1}, w_{1,2}, \ldots, w_{1,2n}$.

Note the map $G_0$ does not affect the first column, so $\tilde{w}_{1,k} = \tilde{w}_{1,k+1}$ whenever $k$ is even.  Therefore $w_{1,k} = -w_{1,k+1}$ for $k$ even.  These are columns $k$ and $k+1$ of $x$ so it follows that $(x,1) \in U$.  The grid of the $w_{j,k}$ has the property that
\begin{equation}
\nu([\begin{array}{cc} w_{j,k} & w_{j+1,k} \\ \end{array}]) 
= -[\begin{array}{cc} w_{j,k+1} & w_{j+1,k+1} \\ \end{array}]
\label{eqnuGrid}
\end{equation}
if $j+k$ is even.  It follows by Lemma \ref{lemTodaConsistency} that 
\begin{equation}
\mu([\begin{array}{cc} w_{j,k} & w_{j,k+1} \\ \end{array}]) 
= [\begin{array}{cc} w_{j+1,k} & w_{j+1,k+1} \\ \end{array}]
\label{eqmuGrid}
\end{equation}
whenever $j+k$ is even.  This, combined with $w_{j,k+2n}=w_{j,k}$ shows that $w_{j,1},\ldots, w_{j,2n}$ are the columns of $F^{j-1}(x,1)$.  Consider $F^{n-1}(x,1)$, whose columns are $w_{n,1},\ldots, w_{n,2n}$.  Either $G_0$ or $G_1$ fixes the $n$th column (depending on if $n$ is even or odd respectively) and it follows that $F^{n-1}(x,1) \in V$.

We have $F^{n-1}(x) \in V$ for all $x \in U$ that arise from a matrix $y \in \mathbb{C}^{2 \times n}$ as above.  It remains to show that generic $x$ are of this form, and that $F^{n-1}:U \to V$ is generically defined and invertible.  All of this would follow from the existence of a grid $w_{j,k}$ of elements of $\mathbb{C}^2$ as above with each instance of $\nu$ and $\mu$ in \eqref{eqnuGrid} and \eqref{eqmuGrid} nonsingular.  This in turn follows from the analogous result for polygon recutting, which in light of Proposition \ref{propTodaCut} is a special case.  

More specifically, construct an array of points $C_{i,j}$ in $\mathbb{R}^2$ indexed by $i+j$ even, $1 \leq i \leq n+1$ as in Lemma \ref{lemJitter}.  This can be thought of as a sequence of $n$ polygons, where the $i$th polygon's vertices are the $C_{i,j}$ and $C_{i+1,j}$.  Define $z_{j,k} \in \mathbb{C}$ for $1 \leq j \leq n$ by
\begin{displaymath}
z_{j,k} = \begin{cases}
C_{j+1,k+1} - C_{j,k}, & j+k \textrm{ even} \\
C_{j,k+1} - C_{j+1,k}, & j+k \textrm{ odd} \\
\end{cases}
\end{displaymath}
Define $w_{j,k} \in \mathbb{C}^2$ by putting $z_{j,k}$ in either the first or second entry depending on the parity of $j$, and putting $\overline{z_{j,k}}$ in the other entry.  By Proposition \ref{propTodaCut} and the properties of the $C_{i,j}$, \eqref{eqnuGrid} and \eqref{eqmuGrid} hold as desired.

\end{proof}

\section{A few conjectures} \label{secConj}
Part of the appeal of the pentagram map and polygon recutting is the simplicity of the underlying constructions that drive the dynamics.  One might expect that making the propagation rules more complicated would lead to less well-behaved systems.  Nonetheless, we have observed that the Devron property continues to hold in large number of settings.  What follows are three systems for which computer experiments strongly suggest that the Devron property is present.

The first system has some resemblance to the pentagram map, but circles through vertices are used instead of diagonals in the construction of new points.  Suppose $A,B,C,D,E$ are five consecutive vertices in a polygon.  Draw the unique circle passing through $A$, $B$, and $C$, and the unique circle passing through $C$, $D$, and $E$.  These two circles intersect at two points, one of which is $C$.  Call the other point of intersection $C'$.  The operation of replacing $C$ with $C'$ in the polygon is called \emph{circle intersection}.  Figure \ref{figCircle} illustrates the construction of $C'$.  We consider bipartite dynamics on even sides polygons.

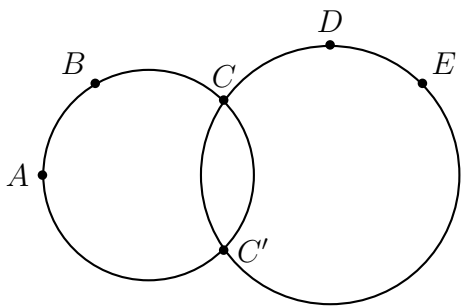
\begin{figure}
\begin{pspicture}(8,4.5)
\rput(0,-1){
\pscircle(2,3){1.414}
\pscircle(4.414,3){1.732}
\psdots(.59,3)(1.29,4.22)(3,2)(3,4)(4.414,4.732)(5.64,4.22)
\uput[l](.59,3){$A$}
\uput[ul](1.29,4.22){$B$}
\uput[u](3,4){$C$}
\uput[r](3,2){$C'$}
\uput[u](4.414,4.732){$D$}
\uput[ur](5.64,4.22){$E$}
}
\end{pspicture}
\caption{The construction used in circle intersection}
\label{figCircle}
\end{figure}

\begin{conj} \label{conjCirc}
Let $A$ be a $2n$-gon, $n\geq 3$ with the property that circle intersection takes half of the vertices to a common point.  Perform bipartite circle intersection on $A$ beginning with the other half of the vertices.  Then after $2n-6$ steps another polygon is reached with the property that circle intersection takes half of its vertices to a common point.
\end{conj}

As an example, suppose $n=3$.  Then $2n-6=0$ so no steps are needed to get from the beginning to the end configuration.  Consider the (degenerate) hexagon $ABCDEF$ pictured in the left of Figure \ref{figCircEx}.  Circle intersection naturally takes place on the Riemann sphere $\mathbb{C} \cup \{\infty\}$.  As such, the ``circles'' through $A,B,C$ and $C, D, E$ are actually lines and their other intersection point is $C'=\infty$.  The same is true for performing circle intersection at $E$ and $A$, so $A'=C'=E'$.  This is the assumption of the conjecture.  The conclusion is that doing circle intersection instead at $B$, $D$, and $F$ produces $B'=D'=F'$.  This follows from a result of Euclidean geometry that asserts that the circle through $B,C,D$, the circle through $D,E,F$, and the circle through $F,A,B$ have a common point of intersection (see the right half of Figure \ref{figCircEx}).

\begin{figure}
\begin{pspicture}(9,4)
\pnode(1,3){A}
\pnode(2,3){B}
\pnode(4,3){C}
\pnode(3,2){D}
\pnode(2,1){E}
\pnode(1.75,1.5){F}
\pspolygon[showpoints=true](A)(B)(C)(D)(E)(F)
\uput[ul](A){$A$}
\uput[u](B){$B$}
\uput[ur](C){$C$}
\uput[dr](D){$D$}
\uput[d](E){$E$}
\uput[dl](F){$F$}

\rput(5,0){
\pnode(1,3){A}
\pnode(2,3){B}
\pnode(4,3){C}
\pnode(3,2){D}
\pnode(2,1){E}
\pnode(1.75,1.5){F}
\psdots(A)(B)(C)(D)(E)(F)
\pscircle(3,3){1}
\pscircle(1.5,2.3125){0.85}
\pscircle(2.458,1.542){0.710}
\uput[ul](A){$A$}
\uput[ul](B){$B$}
\uput[r](C){$C$}
\uput[dr](D){$D$}
\uput[d](E){$E$}
\uput[dl](F){$F$}
}
\end{pspicture}
\caption{An illustration of the case $n=3$ of Conjecture \ref{conjCirc}}
\label{figCircEx}
\end{figure}
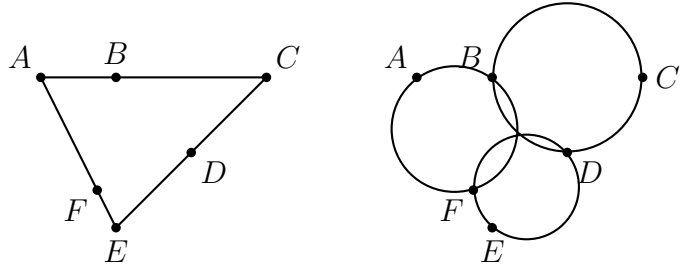

Another natural place to look for the Devron property are the higher dimensional pentagram maps introduced by B. Khesin and F. Soloviev \cite{KS}.  Their systems are similar to those from \cite{GSTV} for which we discussed the Devron property in Section \ref{subsecHigher}.  The difference is that the Khesin Soloviev maps operate on general polygons (sequences of points in $\mathbb{RP}^d$) rather than just corrugated polygons.  For $d=3$, the map takes a polygon $A$ to one who's $i$th vertex is the intersection of the line through $A_{i-1}, A_{i+1}$ and the plane through $A_{i-2},A_i, A_{i+2}$.  Denote this operation $F$.  Below, the \emph{faces} of $A$ refer to the planes passing through triples of consecutive vertices of $A$.

\begin{conj} Let $A$ be a periodic sequence of $2n$ points in $\mathbb{RP}^3$ with $n \geq 4$.  Suppose there are two points $P$ and $Q$ such that the faces of $A$ pass through either $P$ or $Q$ in an alternating fashion.  Let $B=F^{n-3}(A)$.  Then $B_1,B_3,\ldots, B_{2n-1}$ are coplanar and $B_2,B_4,\ldots, B_{2n}$ are also coplanar.
\end{conj}

The last system takes place in $\mathbb{R}^3$ and uses lines as the primitive objects rather than points.  Let $X$ be the set of cyclically ordered $2n$-tuples of lines $L_1,L_2,\ldots, L_{2n}$ in $\mathbb{R}^3$ with the property that for each $i$ the line $L_i$ intersects each of $L_{i-3}$, $L_{i-1}$, $L_{i+1}$, and $L_{i+3}$.  According to Schubert calculus, given four generic lines in 3-space there are exactly two others which intersect all of the original four.  Define a \emph{Schubert flip} to be the operation that replaces $L_i$ with the unique other line intersecting each of $L_{i\pm3}, L_{i\pm1}$.

\begin{conj}
Let $L_1,\ldots, L_{2n}$, $n \geq 6$, be a sequence of lines as before and assume that $L_2=L_4=\ldots=L_{2n}$.  Apply Schubert flips in a bipartite manner beginning with the even lines.  After $2n-7$ steps, another configuration is reached with half of the lines being equal.
\end{conj}


\begin{thebibliography}{99}
\bibitem{A} V. Adler, Cutting of polygons, \textsl{Funct. Anal. Appl.} \textbf{27} (1993), 141--143.
\bibitem{DK} P. Di Francesco and R. Kedem, $T$-systems with boundaries from network solutions, \textsl{Electon. J. Combin.} \textbf{20} (2013), Paper 3.
\bibitem{FZ} S. Fomin and A. Zelevinsky, Cluster algebras I: Foundations, \textsl{J. Amer. Math. Soc.} \textbf{15} (2002), 497--529.
\bibitem{FZ2} S. Fomin and A. Zelevinsky, Cluster algebras IV: Coefficients, \textsl{Compos. Math.} \textbf{143} (2007), 112--164.
\bibitem{GSTV} M. Gekhtman, M. Shapiro, S. Tabachnikov, and A. Vainshtein, Higher pentagram maps, weighted directed networks, and cluster dynamics, \textsl{Electron. Res. Announc. Math. Sci.} \textbf{19} (2012), 1--17.
\bibitem{G} M. Glick, The pentagram map and Y-patterns, \textsl{Adv. Math.} \textbf{227} (2011), 1019--1045.
\bibitem{G2} M. Glick, On singularity confinement for the pentagram map, \textsl{J. Algebraic Combin.} \textbf{38} (2013), 597--635.
\bibitem{GK} A. Goncharov and R. Kenyon, Dimers and cluster integrable systems, arxiv:1107.5588v2.
\bibitem{GRP} B. Grammaticos, A. Ramani, and V. Papageorgiou, Do integrable mappings have the Painlev\'e property?, \textsl{Phys. Rev. Lett.} \textbf{67} (1991), 1825--1828.
\bibitem{I} S. Iwao, Solution of the generalized periodic discrete Toda equation. II. Theta function solution, \textsl{J. Phys. A} \textbf{43} (2010), 155208.
\bibitem{KS} B. Khesin and F. Soloviev, Integrability of higher pentagram maps, \textsl{Math. Ann.} \textbf{357} (2013), 1005-10047.
\bibitem{KNS} A. Kuniba, T. Nakanishi, J. Suzuki, $T$-systems and $Y$-systems in integrable systems, \textsl{J. Phys. A} \textbf{44} (2011), 103001.
\bibitem{LP} T. Lam and P. Pylyavskyy, Total positivity in loop groups, I: Whirls and curls, \textsl{Adv. in Math.} \textbf{230} (2012), 1222--1271.
\bibitem{OST} V. Ovsienko, R. Schwartz, and S. Tabachnikov, The pentagram map: a discrete integrable system, \textsl{Comm. Math. Phys.} \textbf{299} (2010), 409-446.
\bibitem{RR} D. Robbins and H. Rumsey, Determinants and alternating sign matrices, \textsl{Adv. in Math.} \textbf{62} (1986), 169--184.
\bibitem{S0} R. Schwartz, The pentagram map, \textsl{Experiment. Math.} \textbf{1} (1992), 71--81.
\bibitem{S1} R. Schwartz, Desargues theorem, dynamics, and hyperplane arrangements, \textsl{Geom. Dedicata} \textbf{87} (2001), 261--283.
\bibitem{S2} R. Schwartz, Discrete monodromy, pentagrams, and the method of condensation, \textsl{J. Fixed Point Theory Appl.} \textbf{3} (2008), 379--409.
\bibitem{Sp} D. Speyer, Perfect matchings and the octahedron recurrence, \textsl{J. Algebraic Combin.} \textbf{25} (2007), 309--348.
\bibitem{ST} R. Schwartz and S. Tabachnikov, Elementary surprises in projective geometry, \textsl{Math. Intelligencer} \textbf{32} (2010), 31--34.
\bibitem{Z} A. Zamolodchikov, On the thermodynamic Bethe ansatz equations for reflectionless ADE scattering theories, \textsl{Phys. Lett. B} \textbf{253} (1991), 391--394.
\end{thebibliography}
\end{document}